\documentclass{amsart}

\title{Subgroups of the mapping class group of the torus generated by powers of Dehn twists}
\author{Sudipta Kolay}
\address{School of Mathematics \\ Georgia Institute of Technology\\Atlanta, GA 30332, USA}
\email{skolay3@math.gatech.edu}
\date{}

\usepackage[english]{babel}
\usepackage[utf8x]{inputenc}
\usepackage{palatino}

\usepackage[a4paper,top=3cm,bottom=2cm,left=3cm,right=3cm,marginparwidth=1.75cm]{geometry}

\usepackage{amsmath,amsthm,amsfonts,amssymb}
\usepackage{graphicx}
\usepackage[colorlinks=true, allcolors=blue]{hyperref}
\usepackage{spalign}
\usepackage{xfrac} 

\theoremstyle{plain}
\newtheorem{thm}{Theorem}[section]

\newtheorem{prop}[thm]{Proposition}

\newtheorem{claim}[thm]{Claim}
\newtheorem{obs}[thm]{Observation}
\newtheorem{fact}[thm]{Fact}

\theoremstyle{definition}
\newtheorem{defi}[thm]{Definition}

\theoremstyle{remark}
\newtheorem{remark}[thm]{Remark}
\newtheorem{example}[thm]{Example}

\newtheorem{question}{Question}

\begin{document}
\maketitle

\begin{abstract} We study subgroups of the mapping class group of the torus generated by powers generated by powers of Dehn twists.
We give a criterion to show when a collection of powers Dehn twists generates a free group using the ping pong lemma. We show that the subgroup generated by three uniform powers of Dehn twists can be either the whole mapping class group,
a direct product of a free group of rank two with the cyclic group of order two, or a free group of rank at most three. We characterize subgroups generated by a collection of (respectively squares of) Dehn twists, if there is a
pair among so that the geometric intersection number of the corresponding simple closed curves is two (respectively one). We also determine the subgroup generated by uniform powers of all Dehn twists.

\end{abstract}

\section{Introduction}
Mapping class groups of closed oriented surfaces are known to be generated by Dehn twists (Dehn \cite{D}, Lickorish \cite{L}, Humphries \cite{H0}).
A natural question is given a collection of Dehn twists what subgroups of the mapping class group do they generate.
The answer is the free group $F_1\cong \mathbb{Z}$ since Dehn twists (unless otherwise mentioned we will always assume all Dehn twists are about essential\footnote{i.e. not homotopic to a point, a puncture or a boundary component.} simple closed curves) are of infinite order.
A complete answer is known for two Dehn twists, by work of Thurston \cite{T} (and later by  Ishida \cite{I} and Ivanov-McCarthy \cite{IM}), depending on the geometric intersection numbers of the two curves, see Fact \ref{2dt}. The question about $n$ Dehn twists generating the free group $F_n$ has been studied and answered, under suitable hypothesis, by Humphries \cite{H1} and Hamidi-Tehrani \cite{HT}.

Related questions have been studied about group generated by specific families of matrices \cite{Sa,Br,CJR, BM}, parabolic elements in hyperbolic isometry group \cite{S}.

The purpose of this note is to discuss for the torus, what subgroups may be generated by three or more (uniform) powers of Dehn twists. Let us denote by $F_k$ the free group of rank $k$, by $F_\infty$ a free group of countably infinite rank, and by $C_k$ the cyclic group of order $k$.

We can completely answer what subgroups are generated by three uniform powers of Dehn twists.

\begin{thm}\label{thmC}
Given three essential simple closed curves on a torus, and a natural number  $s$, the subgroup of Mod$(\mathbb{T}^2)$ generated by $s$-th powers of Dehn twists about them is either $F_1$, or $F_2$, or $F_3$, or $F_2\times C_2$, or the entire mapping class group Mod($\mathbb{T}^2)= SL(2,\mathbb{Z}$). Moreover, the possibility that the subgroup is $F_2\times C_2$ only arises if $s=2$, and possibility that the subgroup is Mod($\mathbb{T}^2)$ only arises if $s=1$.
\end{thm}
The proof gives an algorithmic way to determine which subgroup it is.

If $G$ is a subgroup generated by $s$-th powers of Dehn twists, we will characterize in Section 7 all subgroups which contain a pair of Dehn twists so that the geometric intersection numbers of the corresponding simple closed curves multiplied by $s$ exactly equals two. 

\begin{thm}\label{thmF} Suppose $x$ and $y$ are simple closed curves in the torus with geometric intersection number $g$. For $sg=2$, the subgroup $G$ of Mod($\mathbb{T}^2)$ generated by a collection of $s$-th powers of Dehn twists including $T_x^s$ and $T_y^s$, is either $F_2$, or Mod($\mathbb{T}^2)=SL(2,\mathbb{Z})$ or $F_2\times C_2$.
 \end{thm}

This theorem follows from Propositions \ref{gino2} and \ref{gino1}, where we analyze the cases $s=1, g=2$ and $s=2,g=1$ separately. It is easy to figure out which of the subgroups $F_2$, or $SL(2,\mathbb{Z})$ or $F_2\times C_2$ we get in a given situation, once we make a change of coordinates, see the comments just following 
Propositions \ref{gino2} and \ref{gino1}.

We also determine the subgroup generated by all $s$-th powers of Dehn twists in the torus.
\begin{thm}\label{thmD}
The subgroup of Mod($\mathbb{T}^2)$ generated by $s$-th powers of all Dehn twists is isomorphic to  Mod($\mathbb{T}^2)=SL(2,\mathbb{Z})$, $F_2\times C_2$, $F_3$,   $F_5$, $F_{11}$ and $F_\infty$ for $s$ equal to one, two, three, four, five and greater than five respectively.
\end{thm}
This result will follow from a slightly more detailed statement in Theorem \ref{thmE} in Section 9.

A key tool to prove the above results is a freeness criterion, for which we will need the notions of  proportional and comparable geometric intersection numbers for a collection of simple closed curves. 
We let $I$ denote a set with cardinality at least two, and $s,s_i$ will be natural numbers for the rest of this article.
  \begin{defi}\label{prop} Given a collection\footnote{There can be repetitions among the $s_i$.} of natural numbers $\{s_i\}_{i\in I}$, we say a collection of distinct essential simple closed curves $\{x_i\}_{i\in I}$ on the torus have $\{s_i\}$-\textbf{proportional} geometric intersection numbers if for all triples of distinct $i,j,k$, the geometric intersection numbers satisfy the inequality $$(x_i,x_j)+(x_j,x_k)+(x_k,x_i)\le   s_j(x_i,x_j)(x_j,x_k).$$
  \end{defi}
The indexing set $I$ will be clear from the context, and so for notational convenience we use the term $\{s_i\}$-{proportional}  instead of $\{s_i\}_{i\in I}$-{proportional}. 

The inequalities in the definition of proportional geometric intersection numbers turn out to be exactly what is required to show the subgroup generated in the following theorem is free using the ping pong lemma.

\begin{thm}\label{thmA}
Suppose $\{s_i\}_{i\in I}$ is a collection of natural numbers, $\{x_i\}_{i\in I}$  is a collection of distinct essential simple closed curves on the torus with $\{s_i\}$-proportional geometric intersection numbers. Then the subgroup
$\langle T^{s_i}_{x_i}|{i\in I}\rangle$ of Mod$(\mathbb{T}^2)$ generated by the $s_i$-th powers of Dehn twists about $\{x_i\}_{i\in I}$ is freely generated by the collection $\{T^{s_i}_{x_i}\}_{i\in I}$.
\end{thm}

\begin{remark}
Note that Theorem \ref{thmA} above easily carry over when $s_i$'s are integers. 
Note that we can replace $T_{x_i}^{s_i}$ with $T_{x_i}^{-s_i}$ if $s_i$ is negative; and we can drop the Dehn twist power $T_{x_i}^{s_i}$ from the collection if $s_i=0$ without changing the subgroup. 
Thus it suffices to just consider the case that $s_i$'s are positive, and we will make that assumption henceforth.
\end{remark}

\begin{defi}\label{cmp}
  We say a collection of distinct essential simple closed curves  $\{x_i\}_{i\in I} $ on the torus are have $\{s_i\}$- \textbf{comparable}  geometric intersection numbers if for all triples of distinct $i,j,k$, the geometric intersection numbers satisfy the inequality $$2(x_i,x_j)\leq s_k(x_i,x_k)(x_k,x_j).$$
  \end{defi}

The motivation for defining $s$-comparable geometric intersection numbers comes from (what we are calling) the Euclidean algorithm (see Section 3), it turns out that if we start out with three curves on the torus, then one can find another collection  of (at most three) curves which have $s$-comparable geometric intersection numbers; so that the subgroups generated by $s$-th powers of Dehn twists about them will be the same.

\begin{remark}
In the above definitions for proportional (respectively comparable) intersection numbers, if all the $s_i$'s are equal say to $s$, then we will say the collection of simple closed curves have \textbf{$s$-proportional} (respectively \textbf{$s$-comparable}) geometric intersection numbers. A lot of our results will be in this scenario, where all the $s_i$'s are the same.

\end{remark}

\begin{remark}
Note that any collection of two (or less) simple closed curves vacuously has $s$-comparable  and $s$-proportional geometric intersection numbers, for any $s$. The reason for not requiring additional conditions for two simple closed curve is we know precisely what relations can occur between two (possibly different) powers of Dehn twists, see Fact \ref{2dt}.
\end{remark}

As we will see in Proposition  \ref{p1}, for $s\neq  2$, if a collection $\{x_i\}_{i\in I}$ has $s$-comparable geometric intersection numbers, then it also has $s$-proportional geometric intersection numbers. Consequently we have:

\begin{thm}\label{thmB}
Suppose $s\in\mathbb{N}\setminus\{2\}$, $\{x_i\}_{i\in I}$  is a collection of distinct essential simple closed curves on the torus with $s$-comparable geometric intersection numbers. Then the subgroup of Mod$(\mathbb{T}^2)$ generated by the $s$-th powers of Dehn twists about $\{x_i\}_{i\in I}$ is freely generated by the collection $\{T^s_{x_i}\}_{i\in I}$.

\end{thm}

 Humphries \cite{H1} and Hamidi-Tehrani \cite{HT} have a similar result for subgroups generated by Dehn twists in general surfaces, under more restrictive conditions, than our notion of comparable/proportional geometric intersection numbers. By restricting to the case of the torus, we can improve the bounds in the hypothesis by using linear dependence among the homology classes of curves corresponding the Dehn twists.\\

In general, if $\{x_i\}_{i\in I}$ are simple closed curves not in comparable/proportional geometric intersection numbers, one would hope to be able to find another  collection of simple closed curves $\{y_j\}_{j\in J}$
(with comparable/proportional geometric intersection numbers), so that the subgroup generated by both the collection of powers of Dehn twists are same.
At present, we are able to systematically do this only for three Dehn twist powers (we call this procedure the Euclidean algorithm, see Section 3) and this is the reason we have the hypotheis of three Dehn twist powers in Theorem \ref{thmC}. However, we will describe a procedure (see Section 8)  combining applications of the ping pong lemma and Hurwitz sliding moves, which in some cases can tell us what the subgroup generated by four or more Dehn twist powers are.

We observe that one obtains the same results as above if we replace the torus with the once punctured torus, since their mapping class groups are isomorphic.
In Section 10, we explain how to translate the above results to subgroups generated by Dehn twist powers in the mapping class group of the one-holed torus, or equivalently, the subgroups generated by half twist powers in the braid group on three strands.

\noindent \textit{Organization:} We discuss background material in Section 2. In Section 3, we introduce Euclidean algorithm for simple closed curves in the torus, and discuss a few examples.
In the following section we apply the ping pong lemma to prove Theorem \ref{thmA}. In Sections 5 and 6, we deal with the cases not taken care of by Theorem \ref{thmA}, and prove Theorems \ref{thmB} and  \ref{thmC}. In the next section, we characterize subgroups which contain a pair of (respectively squares of) Dehn twists whose corresponding curve has geometric intersection number two (respectively one). In Section 8, we explain how to show a group is free using both the ping pong lemma and sliding moves, and discuss a few examples where our earlier techniques failed to determine the subgroup. In Section 9, we study subgroups generated by uniform powers of all Dehn twists and prove Theorem \ref{thmD}. In the final section, we mention implications of the above results to subgroups generated by powers of half twists in the braid group $B_3$.\\

\noindent \textit{Acknowledgements}. The author would like to thank John Etnyre and Dan Margalit for useful discussions and making helpful comments on earlier drafts of this paper. This work is partially supported by NSF grants DMS-1608684 and DMS-1906414.

\section{Preliminaries}

For any oriented surface $S$, primitive elements in $H_1(S)$ correspond to oriented essential simple closed curves on $S$. We will consider simple closed curves up to isotopy, hereafter we assume all simple closed curves are essential.

Dehn twists about a curve $c$ is independent of the orientation of $c$ (the only difference is between right and left handedness). This lets us consider a one-to-two correspondence between Dehn twists in Mod($S$) and homology classes of primitive elements in the first homology of $H_1(S)$, where we send $c$ to $\pm \vec c$ (here we pick an orientation on $c$, and let $ \vec c$ denotes its homology class).

Given two oriented simple closed curve $x,y$, let $(x,y)$ denote their geometric intersection number, and let 
$\langle x,y\rangle$  denote their algebraic intersection number. We will denote by $T_x$ the right handed Dehn twist along the simple closed curve $x$. Note that this homeomorphism induces the following automorphism of the first homology group. $T_x(\vec y)=\vec y+\langle x,y\rangle\vec x$, we note that the subgroup generated by $T_x$ and $T_y$ is the same as the subgroup generated by $T_x$ and $T_z$, where $z=T_x(y)$. So given a collection of simple closed curve $x_1,...,x_k$ (or equivalently collection of primitive homology classes of simple closed curve $\vec{x}_1,...,\vec{x}_k$), we would like to do a suitable change of coordinates so that we get a new collection of simple closed curve $\vec{y}_1,...,\vec{y}_k$ so that the subgroup generated by Dehn twists about the curves are the same, and  $\vec{y}_1,...,\vec{y}_k$ are easier to understand (maybe smaller pairwise geometric intersection numbers). We would like to caution the reader that this is not about choosing a different basis of the lattice generated by the original collection of homology classes, indeed some of the new basis vectors may not be primitive elements, and hence would not correspond to simple closed curves.

We recall known results about subgroup generated by two Dehn twists.
\begin{fact} \label{2dt}

(see \cite[Section 3.5]{FM})
 Let us consider the subgroup $G$ of Mod($S$) generated by Dehn twists powers $T_x^s$ and $T_y^t$, with $0<s\leq t$. If $S$ not the torus or once punctured torus, we have:

\begin{enumerate}
    \item $G\cong F_1 $ if and only if $(x,y)=0$ (i.e. $x$ and $y$ are the same upto orientation, and so  $T_x=T_y$),
     \item $G\cong  B_3 $  if and only if $(x,y)=1=s=t$,
     \item $G\cong \langle a,b| abab=baba\rangle$ if and only if $(x,y)=1$ and  $s=1,t=2$,
\item $G\cong \langle a,b| ababab=bababa\rangle$ if and only if $(x,y)=1$ and $s=1,t=3$, 
      \item $G\cong F_2$ otherwise.
\end{enumerate} 
If $S$ is the torus or once punctured torus, we have:
\begin{enumerate}
    \item $G\cong F_1 $ if and only if $(x,y)=0$ (i.e. $x$ and $y$ are the same upto orientation, and so  $T_x=T_y$),
     \item $G\cong SL(2,\mathbb{Z})$ if and only if $(x,y)=1=s=t$,
    \item $G\cong \langle a,b| abab=baba,(b^2a)^4=1\rangle$ if and only if $(x,y)=1$ and $s=1,t=2$,
\item $G\cong \langle a,b| ababab=bababa,(b^3a)^3=1 \rangle$ if and only if $(x,y)=1$ and $s=1,t=3$,
     \item $G\cong F_2$ otherwise.
\end{enumerate} 

\end{fact}

Henceforth (apart from Section 10) we will take the surface to be torus $\mathbb{T}^2$. Let us consider three Dehn twists $T_x$, $T_y$ and $T_z$ in Mod($\mathbb{T}^2$). Since $H_1(\mathbb{T}^2)\cong \mathbb{Z}^2$ (we will always use $\mathbb{Z}$ coefficients), we know that $\vec x, \vec y$ and $\vec z$ are linearly dependent.
In order to understand the subgroup $G$ of Mod($\mathbb{T}^2$) generated by $T_x^s$, $T_y^s$ and $T_z^s$, we will try to understand the subset $H^s_G$ of first homology corresponding to  Dehn twists in $G$. More formally, let us define 
$$H^s_G:=\{\vec x\in  H_1(\mathbb{T}^2)| \vec x \text{ is primitive and } T_x^s\in G \}$$

We note that $H^1_G$ is consists of all primitive elements of $H_1(\mathbb{T}^2)$ if and only if $G=$Mod($\mathbb{T}^2)$.

\subsection{Comparable and proportional geometric intersection numbers}

While neither the notion of $s$-comparable or $s$-proportional directly implies the other, as we will see below overlap in a lot of cases.

\begin{obs}\label{ob1}
For any natural numbers $a,b,s$, we have the inequality $sab\geq 2(a+b)$, unless we are in the following cases:
\begin{enumerate}
    \item $s=1$, $a=3$ and $b\in \{3,4,5\}$,
    \item $s=1$, $b=3$ and $a\in \{3,4,5\}$,
    \item $s=1$, $a<3$ or $b<3$,
    \item $s=2$, $a=1$ or $b=1$,
    \item $s=3$, $a=1$ and $b=1$.
    
\end{enumerate}

\end{obs}

\begin{proof}
 Without loss of generality let us assume that $a\leq b$.
 If $a\geq 4$, then $sab\geq 4b\geq 2(a+b)$, as $s\geq 1$. 
 If $a=3$, and $s=1$, then $sab\geq2(a+b)$, which is equivalent to $b\geq 2a =6$.
 If $a\geq 2$, then $sab\geq 2sb\geq s(a+b)$, and so the inequality holds as long as $s\geq2$.
 If $a= 1$, $b>1$, the inequality holds for $s\geq 3$.
 If $a=1$, $b=1$, the inequality holds for $s\geq 4$.
\end{proof}

\begin{remark}\label{r9}
If a collection $\{x_i\}_{i\in I}$ of simple closed curve have  $\{s_i\}_{i\in I} $-{comparable} geometric intersection numbers, and for any triple $i,j,k$ we do not get any of the exceptional cases in Observation \ref{ob1} by taking $a=(x_j,x_k)$, $b=(x_k,x_i)$ and $s=s_j$, then the collection $\{x_i\}_{i\in I}$ have $\{s_i\}$-proportional geometric intersection numbers, since
$$(x_i,x_j)+(x_j,x_k)+(x_k,x_i)\leq \frac{s_j}{2}(x_i,x_j)(x_j,x_k)+\frac{s_j}{2}(x_i,x_j)(x_j,x_k)=s_j(x_i,x_j)(x_j,x_k).$$
\end{remark}

In later sections, when we are dealing with uniform $s$-th powers of three Dehn twists,
we will see by the Euclidean algorithm we can find (at most) three simple closed curves which have $s$-comparable geometric intersection numbers, and the $s$-th powers of  Dehn twists about them generate the original subgroup. In most cases, it will follow that the new collection will also have $s$-proportional geometric intersection numbers, and by the  Theorem \ref{thmA} it will follow that $s$-th powers of their Dehn twists generate a free group. We would analyze the exceptional cases separately to figure out what subgroups  $s$-th powers of their Dehn twists generate.


\section{Euclidean Algorithm}
Let us now consider two simple closed curves $x,y$ in the torus with geometric intersection number $n\geq 2$.

By change of coordinates we may assume that the homology class of $x$ is $\spalignmat{1; 0}$, and then the homology class of $y$ (by choosing the sign appropriately) is $\spalignmat{a; n}$, where $a$ is arbitrary. Moreover, by composing with an appropriate power of $T_x$, and using the division algorithm, we may assume that $0<a<n$ (the case $a=0$ does not occur, since it does not correspond to a primitive element). \\

Given three Dehn twists powers $T_x^s$, $T_y^t$ and $T_z^u$, we we describe below an algorithmic procedure to find at most three simple closed curves whose powers of Dehn twists generate the same subgroup. Since this process is similar to the classical Euclidean algorithm to find the greatest common divisor of two integers, we will call this procedure an Euclidean algorithm.

Suppose we are given three pairwise distinct simple closed curve $x,y,z$, and by using linear dependence we write
$\vec z$ as a linear combination of $\vec x$ and  $\vec y$;  $\vec z=\alpha \vec x+ \beta \vec y$ where $\alpha, \beta\in \mathbb{Q}$. Let us see the effect of the $k$-th power of Dehn twist $T_x$. 
$$T_x^k \vec z=\vec z+ k\langle x,z\rangle \vec x =\alpha \vec x+ \beta \vec y+ k\beta \langle x,y\rangle\vec x = (\alpha + k\beta \langle x,y\rangle)\vec x  + \beta \vec y$$
By the (slightly modified) division algorithm, we can find $k\in s\mathbb{Z}$ so the absolute value of the $\vec x$ coefficient of $\vec z_1$ is at most $\frac{s|\beta|(x,y)}{2}= \frac{s(z,x)}{2}$, where we denote by $z_1$ the simple closed curve $T_x^k z$.
Note $(x,z_1)=(x,z)$, and $(y,z_1)<(y,z)$ if $k\neq 0$ (and same otherwise, since then $z_1=z$), and moreover $T_{z_1}^u =T_x^s \circ T_z^u \circ T_x^{-s}$, thus the subgroup generated by $T_x^s, T_y^t, T_z^u$ is the same as the subgroup generated by $T_x^s, T_y^t, T_{z_1}^u$.

So we conclude that if $|\alpha|> \frac{s|\beta|(x,y)}{2}$ , or equivalently $2(y,z)>s (x,y)(z,x)$,
we can apply some power of $T_x^s$ to $z$, to get a new collection of simple closed curve whose respective powers of Dehn twists generate the same group, and one of the pairwise geometric intersection numbers get smaller. We can interchange the roles of $x,y$ and $z$, and keep repeating the process till we get new collection of simple closed curve $x',y', z'$ so that either $$2(x',z')\leq t(x',y')(y',z'), 2(x',y')\leq u(x',z')(y',z'), 2(y',z')\leq s (x',y')(x',z'),$$
i.e. $x',y',z'$ are have $s$-comparable geometric intersection numbers, or at some point some pairwise geometric intersection number became zero, in which case we replace the two Dehn twist powers $T_w^a$ and $T_w^b$ with the Dehn twist $T_w^g$, where $g=\gcd(a,b)$.

Since each of the geometric intersection number is a non-negative integer, and in each iteration one of the pairwise geometric intersection number (or the number of Dehn twist powers) goes down, we see that this process  has to stop in finite number of steps. 

In case $s=t=u$, the result of the above algorithm will be collection of simple closed curve with $s$-comparable geometric intersection numbers has the property that the $s$-th powers of Dehn twists along them generate the same subgroup as the original one. One may hope to find a similar algorithm for the case of more than three simple closed curve and then we could figure out exactly which subgroup a collection of $s$-th powers of Dehn twists generate. See Example \ref{tricky} below to see why this may be tricky.

\subsection{Examples}

We discuss a few examples, illustrating the Euclidean algorithm, and assuming Theorem \ref{thmB}, figure out which subgroups are generated. 

\begin{example} \label{eg1}
If $\vec x= \vec y=\vec z= \spalignmat{1; 0} $, then the subgroup generated by $T_x$, $T_y$, $T_z$ is isomorphic to $F_1$.
\end{example}

\begin{example} \label{eg2}  
If $\vec x= \spalignmat{1; 0},\vec y= \spalignmat{1; 3},\vec z= \spalignmat{-11; 3}$, we can set
 $$\vec x_1= \spalignmat{1; 0},\vec y_1= \spalignmat{1; 3},\vec z_1=T_x^{4}\vec z= \spalignmat{1; 3};$$
 Since $\vec y_1=\vec z_1$, we only consider $x_1$ and $y_1$, which have geometric intersection number 3, and so the subgroup generated by $T_x$, $T_y$ and $T_z$ is isomorphic to $F_2$ by Fact \ref{2dt}.
\end{example}

\begin{example} \label{eg3} If $\vec x= \spalignmat{1; 0},\vec y= \spalignmat{4; 3},\vec z= \spalignmat{1; 6}$, we can set
    $$\vec x_1= \spalignmat{1; 0},\vec y_1=T_x^{-1}\vec y= \spalignmat{1; 3},\vec z_1= \spalignmat{1; 6};$$
    $$\vec x_2= \spalignmat{1; 0},\vec y_2= \spalignmat{1; 3},\vec z_2=T_y^{-1}\vec z= \spalignmat{-2; 3}$$
    Now $x_2,y_2,z_2$ have 1-comparable geometric intersection numbers, and by the following section subgroup generated by $T_x$, $T_y$ and $T_z$ is isomorphic to $F_3$ by Theorem \ref{thmB}.

\end{example}

\begin{example} \label{eg4} If $\vec x= \spalignmat{1; 0},\vec y= \spalignmat{7; 3},\vec z= \spalignmat{1; 4}$, we can set
     $$\vec x_1= \spalignmat{1; 0},\vec y_1=T_x^{-2}\vec y= \spalignmat{1; 3},\vec z_1= \spalignmat{1; 4};$$
     $$\vec x_2= \spalignmat{1; 0},\vec y_2= \spalignmat{1; 3},\vec z_2=T_y^{-1}\vec z_1= \spalignmat{0; 1};$$
     $$\vec x_3= \spalignmat{1; 0},\vec y_3= T_z^{-3}\vec y_2=\spalignmat{1; 0},\vec z_3= \spalignmat{0; 1};$$
      Since $\vec y_3=\vec x_3$, we only consider $x_3$ and $z_3$, which have geometric intersection number 1, and so the subgroup generated by $T_x$, $T_y$ and $T_z$ is Mod$(\mathbb{T}^2)$ by Fact \ref{2dt}.

\end{example}

  \begin{example}\label{tricky} If $\vec x= \spalignmat{1; 0},\vec y= \spalignmat{1; 3},\vec z= \spalignmat{1; 10},\vec w= \spalignmat{3; 17}$;we see that the three triples of simple closed curve $x,y,z$; $x,w,z$  and $y,z,w$ have 1-comparable geometric intersection numbers, and $x,y,w$ are not. Consider
  $$\vec x_1= \spalignmat{1; 0},\vec y_1= \spalignmat{1; 3},\vec z_1= \spalignmat{1; 10},\vec w_1= T_y^{-1} \vec w= \spalignmat{3; 17}-8\spalignmat{1; 3}=\spalignmat{-5; -7};$$
  Now we see that $x_1,y_1,w_1$ have 1-comparable geometric intersection numbers but $y_1,z_1,w_1$ do not have 1-comparable geometric intersection numbers.
   So given four or more simple closed curve it seems unlikely that one can come up with a similar algorithm in general, because making some pairwise geometric intersection number smaller by updating an simple closed curve with some power of a Dehn twist, may increase some other pairwise geometric intersection number. However we will show in Example \ref{eg7} that  $T_x,T_y,T_z$ and $T_w$ generate the free group $F_4$, using a procedure which combines application of the ping pong lemma with Hurwitz sliding moves.
  \end{example}

\section{Ping Pong}

To prove the freeness of the subgroup in Theorem 1, we will use the ping pong lemma, which is a criterion to show a group is free by understanding how it acts on a set.

\begin{fact}(Ping pong lemma \cite[Lemma 3.15]{FM}) Suppose a group $G$ acts on a set $X$. Suppose we have a subset $\{g_i\}_{i\in  I}$  (with $I$ having cardinality at least two) of $G$ and there are disjoint non-empty subsets $\{X_i\}_{i\in  I}$ of $X$ so that for any $p\neq 0$, we have $g_i^p(X_j)\subseteq X_i$ whenever $i\neq j$. Then the set $\{g_i\}_{i\in  I}$ freely generate a subgroup of $G$.
\end{fact}

The lemma is usually stated and proved when $I$ is a finite set. However note that if there is a relation, it only involves finitely many elements in the collection $\{g_i\}_{i\in  I}$, and so the above version of the ping pong lemma reduces to the case where the indexing set is finite.

\begin{proof} [Proof of Theorem \ref{thmA}]
For $ i\in I$, consider the sets 
$$ N_i:= \{c\text{ simple closed curve in } \mathbb{T}^2| (c,x_i)< (c,x_j)\text{ for all }j \in I \setminus \{i\} \}$$
We observe that the sets are disjoint, and non-empty (since $x_i\in N_i$).
 By the ping pong lemma, the theorem will be proved, if we show for $p\in \mathbb{Z}\setminus \{0\}$, any two distinct indices $i,j$, and any $c\in N_i$, $T_{x_j}^{s_jp}(c)\in N_j$. To establish this, we need to show $(T_{x_j}^{s_jp}(c),x_j)<(T_{x_j}^{s_jp}(c),x_k)$ for any $k\neq j$.

For notational convenience, suppose $x=x_i$, $y=x_j$, $z=x_k$, $N_x=N_i,N_y=N_j,N_z=N_k$ and $s=s_j$. Since $x,y$ are distinct simple closed curves; $\vec x$ and $\vec y$ form a basis for $\mathbb{Q}^2$ over the rational numbers, we can write $\vec z$ as a rational linear combination 
$$\vec z=\alpha \vec x+ \beta \vec y.$$

Let $c\in N_x$, want to show $T_y^{sp}c\in N_y$, so we need $(T^{sp}_y c,x)>(T^{sp}_y c,y)$ and $(T^{sp}_y c,z)>(T^{sp}_y c,y)$.
Note that $ T^{sp}_y\vec c=\vec c+sp\langle y,c\rangle \vec y$,
\text{and hence   }$\langle T^{sp}_y c,x\rangle =\langle c,x\rangle +s p\langle y,c\rangle  \langle y,x\rangle.$
\text{Thus we see    }$$(T^{sp}_y c,x)\geq -(c,x)+ s(y,c)(y,x)\geq -(c,x)+2(c,y) > (c,y)=(T^{sp}_y c,y)$$

Now  we want to show $(T^{sp}_y c,z)>(T^{sp}_y c,y)=(c,y)$. We observe that $\langle T^{sp}_y c,z\rangle=\langle c,z\rangle + sp\langle y,c\rangle \langle y,z\rangle$, and therefore
$$(T^{sp}_y c,z)\geq -(c,z)+ s(c,y)(y,z)\geq -|\alpha| (c,x)-|\beta| (c,y)+ s(c,y)(y,z)>(c,y)(s(y,z)-|\alpha|-|\beta|).$$

We want the rightmost expression in the above chain of inequalities to be at least $(c,y)$, and for this it suffices to have 
$s(y,z) \geq|\alpha|+|\beta|+1 $, or equivalently by multiplying both sides with $(x,y)$

$$ s(x,y)(y,z)\geq |\alpha|(x,y)+|\beta|(x,y)+(x,y)= (y,z)+(x,z)+(x,y) $$

Recalling our notation for $x,y,z$ and $s$, we see this is the same inequality as in the hypothesis.

\end{proof}

\begin{example}If $\vec x= \spalignmat{1; 0}$,  $\vec y= \spalignmat{0; 1}$,  $\vec z= \spalignmat{1; 1}$. By Theorem \ref{thmA}  we see that $T_x^p, T_y^q$ and $T_z^r$ generates a free group $F_3$ as long as $p,q,r$ are all at least 3.
\end{example}

\section{The remaining cases to consider for Theorems \ref{thmB} and \ref{thmC} }

In this section we will consider all the cases we cannot directly apply Observation \ref{ob1} and Theorem \ref{thmA} to prove Theorems \ref{thmB} and \ref{thmC}. Let us suppose $x,y, z$ are three simple closed curves with $s$-comparable geometric intersection numbers.  Without loss of generality let us assume that $(x,y)\leq (y,z)\leq (z,x)$.

\begin{obs} For $x,y,z$ and $s$ as above, if $s(x,y)=2$, then $(y,z)=(x,z)$.
\end{obs}

This follows directly from the inequalities
$$2(x,z)\leq s(x,y)(y,z), \quad 2(y,z)\leq s (x,y)(x,z).$$ 

We now observe that $s(x,y)=2$ can only occur in the following two situations:
\begin{enumerate}
    \item $s=1$ and $(x,y)=2$:  By a change of coordinates, we may assume that $\vec x= \spalignmat{1; 0}$ and  $\vec y= \spalignmat{1; 2}.$ Assuming $\vec z= \spalignmat{p; q}$, we see that $(x,z)=|q|$, and $(y,z)=|2p-q|$. For $(x,z)=(y,z)$, we must either have $q=2p-q$ ($\Leftrightarrow p=q$), or $-q=2p-q$  ($\Leftrightarrow p=0$). In either case we see that for $\vec z$ to be primitive we need to have $|q|=1$ contradicting our assumption that $x,y,z$ are 1 -comparable ; and also $(x,y)<(x,z)$. Moreover, we know the subgroup generated in this case is Mod$(\mathbb{T}^2)$.

    \item $s=2$ and $(x,y)=1$: By change of co-ordinates we may assume that $\vec x=\spalignmat{1; 0}$ and  $\vec y=\spalignmat{0; 1}$, and let us suppose $\vec z=\spalignmat{p; q}$. In this case we have $(x,y)=1$, $(y,z)=|p|$ and $(x,z)=|q|$.
Thus $|p|=|q|$, and so the only way $\vec z$ can be primitive is $|p|=|q|=(x,z)=(y,z)=1$.
\end{enumerate}

By the above discussion and Observation \ref{ob1}, 
the remaining cases to check are:
\begin{enumerate}
    \item  $(x,y)=3$ and $(y,z)\in \{3,4,5\}$; and $s=1$,
    \item  $(x,y)=1=(y,z)$; and $s\in\{2,3\}$. 
\end{enumerate}

We will deal with them in the following two subsections.

\subsection{Three Dehn twists}

By change of coordinates, let us assume $\vec y=\spalignmat{1;0}.$ In what follows, we will talk about choices for a second vector $\vec w$ with geometric intersection number with $\vec y$ being $n$. By change of coordinates, and up to taking negatives and remainder modulo $\vec y$ (i.e. apply a power of Dehn twist about $y$), for $w$ a simple closed curve with $(w,y)=n$ we will assume without loss of generality that $\vec w=\spalignmat{a; n}$, with $1< a < n$. Given some $x,y,z$ with $\vec y=\spalignmat{1; 0},$ it may not be possible just by a change of coordinates to send $x$ to $x_1$
and $z$ to $z_1$ fixing $y$, where  $\vec x_1$ and $\vec z_1$ have the form mentioned in the previous line, however the subgroup generated by $x,y,z$ is the same as the subgroup generated by $x_1,y,z_1$.

\begin{itemize}
    \item Let us consider the case $(x,y)=(y,z)=3$.\\
    The choices for $\vec x$ and $\vec z$ are $\spalignmat{1; 3}$ and $\spalignmat{2; 3}$. Note that if $\vec x=\vec z $, then $T_x=T_z$ and $T_x,T_y$ and $T_z$ generate the free group $F_2$.
\begin{center}
  \begin{tabular}{ |c | c|c|c| c| c |c|}
    \hline
    $\vec x$ & $\vec  z$ & $(z,x)$ & $(z,x)$ largest &  $ (x,y)+(y,z)+(z,x)$ & 1-comparable  &  1-proportional \\ \hline
    $\spalignmat{1; 3}$ & $\spalignmat{2; 3}$ &   $3$ & Yes & $3+3+3=9$  &Yes & Yes  \\ \hline
     $\spalignmat{2; 3}$ & $\spalignmat{1; 3}$  &   $3$ & Yes & $3+3+3=9$ & Yes & Yes\\\hline
  \end{tabular}
\end{center}
We conclude that in the case $(x,y)=(y,z)=3$ the collection $x,y$ and $z$ are 1-proportional if they are 1-comparable.

\item Let us consider the case $(x,y)=3$ and $(y,z)=4$.\\
The choices for $\vec x$ are $\spalignmat{1; 3}$ and $\spalignmat{2; 3}$, and the choices for $\vec z$ are $\spalignmat{1; 4}$ and $\spalignmat{3; 4}$.

\begin{center}
  \begin{tabular}{ |c | c|c|c| c| c |c|}
    \hline
    $\vec x$ & $\vec  z$ & $(z,x)$ & $(z,x)$ largest &  $ (x,y)+(y,z)+(z,x)$ & 1-comparable  &  1-proportional \\ \hline
    $\spalignmat{1; 3}$ & $\spalignmat{1; 4}$ &   $1$ & No & $3+4+1=8$  & No & No  \\ \hline
     $\spalignmat{1; 3}$ & $\spalignmat{3; 4}$   & $5$ & Yes &  $3+4+5=12$ & Yes & Yes\\\hline
    $\spalignmat{2; 3}$ & $\spalignmat{1; 4}$    & $5$ & Yes &  $3+4+5=12$ & Yes & Yes\\\hline
    $\spalignmat{2; 3}$ & $\spalignmat{3; 4}$ &   $1$ & No & $3+4+1=8$  & No & No  \\ \hline
  \end{tabular}
\end{center}
We conclude that in the case $(x,y)=3$ and $(y,z)=4$ the collection $x,y$ and $z$ are 1-proportional if they are 1-comparable.\\

\item Let us consider the case $(x,y)=3$ and $(y,z)=5$.\\
The choices for $\vec x$ are $\spalignmat{1; 3}$ and $\spalignmat{2; 3}$, and for $\vec z$ are $\spalignmat{1; 5}$, $\spalignmat{2; 5}$, $\spalignmat{3; 5}$ and $\spalignmat{4; 5}$.

\begin{center}
  \begin{tabular}{ |c | c|c|c| c| c |c|}
    \hline
    $\vec x$ & $\vec  z$ & $(z,x)$ & $(z,x)$ largest &  $ (x,y)+(y,z)+(z,x)$ & 1-comparable  &  1-proportional \\ \hline
   $\spalignmat{1; 3}$ & $\spalignmat{1; 5}$ &   $2$ & No & $3+5+2=10$  & No & No  \\ \hline
   $\spalignmat{1; 3}$ & $\spalignmat{2; 5}$ &   $1$ & No & $3+5+1=9$  & No & No  \\ \hline
   $\spalignmat{1; 3}$ & $\spalignmat{3; 5}$ &   $4$ & No & $3+5+4=12$  & Yes & Yes  \\ \hline
   $\spalignmat{1; 3}$ & $\spalignmat{4; 5}$ &   $7$ & Yes & $3+5+7=15$  & Yes & Yes  \\ \hline
   $\spalignmat{2; 3}$ & $\spalignmat{1; 5}$ &   $7$ & Yes & $3+5+7=15$  & Yes & Yes \\ \hline
   $\spalignmat{2; 3}$ & $\spalignmat{2; 5}$ &  $4$ & No & $3+5+4=12$  & Yes & Yes  \\ \hline
   $\spalignmat{2; 3}$ & $\spalignmat{3; 5}$ &   $1$ & No & $3+5+1=9$  & No & No  \\ \hline
   $\spalignmat{2; 3}$ & $\spalignmat{4; 5}$ &   $2$ & No & $3+5+2=10$  & No & No  \\ \hline
  \end{tabular}
\end{center}
We conclude that in the case $(x,y)=3$ and $(y,z)=5$ the collection $x,y$ and $z$ are 1-proportional if they are 1-comparable.\\

\end{itemize}


\subsection {Two geometric intersection numbers are 1}

Suppose we have three simple closed curves $x,y$ and $z$, with $(x,y)=1=(y,z)$. If $x,y,z$ are $s$-comparable, then $2(x,z)\leq s(x,y)(y,z)=s$. For $s\in \{2,3\}$, the possibilities for $(x,z)$ are 0 and 1. If $(x,z)=0$, then
we have $T_x=T_z$, and hence $T_x^s,T_y^s$ and $T_z^s$ generate the free group $F_2$. So we only need to consider the case $(x,z)=1$.

For $s=3$, we see that $x,y,z$  have 3-proportional geometric intersection numbers, and so $T_x^3$, $T_y^3$ and $T_z^3$ generate the free group $F_3$.

For $s=2$, we see that $x,y,z$ do not have 3-proportional geometric intersection numbers, so we need a more careful analysis.
By change of coordinates we may assume that the homology classes of the curves are: $\vec x= \spalignmat{1; 0}$,  $\vec y= \spalignmat{0; 1},$ and then $\vec z$ has to be
$\spalignmat{\pm 1; \pm 1}$.
Note that $$T_y^2\spalignmat{ 1;  1}=\spalignmat{ 1;  -1},
T_x^2\spalignmat{ 1;  -1}=\spalignmat{ -1;  -1}, T_y^2\spalignmat{ -1;  -1}=\spalignmat{ -1;  1}, T_x^2\spalignmat{ -1;  1}=\spalignmat{ 1;  1}.$$
By applying sufficient powers of $T_x^2$ and $T_y^2$ to $z$, we may assume that  $\vec z= \spalignmat{1; 1}$.

\begin{claim}\label{csquares} With $x,y,z$ as above, the subgroup generated by $T_x^2$, $T_y^2$ and $T_z^2$ is $F_2\times C_2$. 
\end{claim}
\noindent The claim is proved in \cite[Section 3]{BM} in the language of matrices, and we include a proof for completeness.
\begin{proof}
Since $T_x^{2} T_y^{2}(\vec z)=-\vec z $,  $T_z$ commutes with $T_x^{2} T_y^{2}$, and thus $T_x^{2} T_y^{2}$ must be some power of $T_z$ composed with the hyperelliptic involution $\iota$. In fact, we see that $T_x^{2} T_y^{2}=\iota T_z^{-2}$ by looking at the action on $\vec y$. Thus the subgroup generated by $T_x^2$, $T_y^2$ and $T_z^2$ is the same as the subgroup generated by $T_x^2$, $T_y^2$ and $\iota$. We know that $T_x^2$ and $T_y^2$ generate the free group $F_2$, which does not contain $\iota$ (otherwise there would be relations). Also, $\iota$ is a central element in Mod$(\mathbb{T}^2)$ and generates the cyclic group $C_2$. It follows that the subgroup generated by $T_x^2$, $T_y^2$ and $T_z^2$ is an internal direct product of the above two subgroups, and the result follows.
\end{proof}

\section{Combining Euclidean algorithm with the case analysis}
In this section we combine the ideas from previous sections to prove Theorems \ref{thmB} and \ref{thmC}.

\subsection{Comparable versus Proportional Geometric Intersection numbers}
We saw in the last section that there are simple closed curves $x,y,z$  with
2-comparable geometric intersection numbers, but do not have 2-proportional geometric intersection numbers; where
$\vec x=\spalignmat{1; 0}$,  $\vec y=\spalignmat{0; 1}$ and $\vec z= \spalignmat{ 1; 1}$.

One can also see that there are natural numbers $a,b,c,s$ which satisfy the $s$-comparable inequalities ($2a\leq bc$, $2b\leq ca$, $2c\leq ab$) but not the $s$-proportional inequalities ($a+b+c\leq s \min \{ab,bc,ca\}$), such as:
\begin{enumerate}
  \item  $a=2,b=n,c=n$ (for any $n>2$) and $s=1$; 
  \item  $a=1,b=n,c=n$ (for any $n>2$) and $s=2$;
  \item $a=3,b=4,c=6$ and $s=1$.
\end{enumerate}
However, as explained below, we cannot have three simple closed curves whose pairwise geometric intersection numbers realize the above $a,b,c$.
Suppose $x,y,z$ are simple closed curves with $s$-comparable geometric intersection numbers;
and as in the previous section we will assume $(x,y)\leq (y,z)\leq (z,x)$ and $\vec y=\spalignmat{1; 0}$.
For $s=1$, $\vec {x_1}=k(x,y)\vec y+ \vec x$ and $\vec {z_1}=l(z,y)\vec y+ \vec z$,
where $k,l\in\mathbb{Z}$, and $x_1$ and $z_1$ are the various vectors appearing in the tables of Subsection 6.1. It follows that
$$\langle x_1,z_1\rangle=k(x,y)\langle y,z \rangle+l(y,z)\langle x,y \rangle+\langle x,z\rangle= m(x,y)(y,z)+\langle x,z\rangle$$ for some $m\in \mathbb{Z}$.
Consequently $(x,z)$ is $|(x_1,z_1)\pm m(x,y)(y,z)|$. Looking at the tables, and recalling that $2(x,z)\leq (x,y)(y,z)$; it follows that $(x,z)$ is either one of the geometric intersection numbers appearing in that table; or  $(x,y)(y,z)$ minus such an intersection number appearing in the table. In all these cases it follows that $x,y,z$ have 1-proportional geometric intersection numbers whenever they have 1-comparable geometric intersection numbers. So we conclude:

\begin {prop}\label{p1} For $s\in \mathbb{N}\setminus\{2\}$, a collection of distinct simple closed curves have $s$-proportional geometric intersection numbers whenever they have $s$-comparable geometric intersection numbers.
\end{prop}

\begin{proof}
It follows from Remark \ref{r9} that if we are not in any of the exceptional cases in Observation \ref{ob1}, that  $s$-comparable implies $s$-proportional geometric intersection numbers. We considered the case $s=1$ in Section 6.2 and the preceding paragraph; and we looked at the case $s=3$ in Section 6.2 and again saw that the same conclusion holds (while there are natural numbers that violate the inequalities, they do not arise as geometric intersection number of simple closed curves).
\end{proof}

\begin{remark}
Note that we can have simple closed curves which have $s$-proportional geometric intersection numbers but not $s$-comparable geometric intersection numbers, for example:
\begin{enumerate}
    \item $\vec x=\spalignmat{1; 0}$,  $\vec y=\spalignmat{4; 5},\vec z= \spalignmat{ 1; 5} $ and $s=1$.
     \item $\vec x=\spalignmat{1; 0}$,  $\vec y=\spalignmat{4; 5},\vec z= \spalignmat{ -2; 5} $ and $s=2$.
\end{enumerate}
It is not too hard to construct lots of such examples as long as all pairwise geometric intersection numbers are at least 5. Thus we see that neither notions of $s$-proportional or $s$-comparable is weaker than the other.
\end{remark}

\begin{proof} [Proof of Theorem \ref{thmB}] This follows immediately from Proposition \ref{p1} and Theorem \ref{thmA}. 
\end{proof}

\subsection{Applications of Theorem \ref{thmB}}

\begin{example} For any natural number $p\geq 4$, consider $\vec x_k=\spalignmat{1; pk}$, and let $x_k$ denote the corresponding simple closed curve.  It follows from Observation \ref{ob1} that  the collection $\{ x_k\}_{k\in \mathbb{Z}}$ have 1-comparable geometric intersection numbers, and consequently, $\{ T_{x_k}\}_{k\in \mathbb{Z}}$ freely generate a subgroup isomorphic to $F_\infty$.

\end{example}

\begin{example} \label{eg} Consider $\vec y_k=\spalignmat{k; 1}$, and let $y_k$ denote the corresponding simple closed curve. For any $s\geq 4$, it follows from Observation \ref{ob1} that  the collection $\{ y_k\}_{k\in \mathbb{Z}}$ have $s$-comparable geometric intersection numbers, and consequently, $\{ T^s_{y_k}\}_{k\in \mathbb{Z}}$ freely generate a subgroup isomorphic to $F_\infty$.

\end{example}

\subsection{Proof of Theorem \ref{thmC}}

\begin{proof}
Starting with three essential simple closed curves and $s\in \mathbb{N}$, we can apply the Euclidean algorithm to get at most three simple closed curves with $s$-comparable geometric intersection numbers. 

In cases we have three simple closed curves and they also have $s$-proportional geometric intersection numbers, we know that the subgroup generated is $F_3$. By the cases we considered above, we saw the only other possibility is $F_2\times C_2$ (only possible if $s=2$).

In case the Euclidean algorithm ends with just two simple closed curves, then by Fact \ref{2dt}
 the subgroup can either be $F_2$, or the entire Mod$(\mathbb{T}^2)$ (only possible if $s=1$). And if the algorithm ends with just one simple closed curve the subgroup is $F_1$.

Examples \ref{eg1}, \ref{eg2}, \ref{eg3}, \ref{eg4}, and Claim \ref{csquares} show all the possibilities in the statement of Theorem \ref{thmC} do occur.

\end{proof}

\section {Small powers and geometric intersection numbers}

We know that if a subgroup $G$ of Mod($\mathbb{T}^2)$ contains two Dehn twists $T_x$ and $T_y$ with $(x,y)=1$, then $G$ is all of Mod($\mathbb{T}^2)$. In a similar vein,  in this section we will characterize subgroups of Mod($\mathbb{T}^2)$ generated by $s$-th powers of Dehn twists, if it contains $T_x^s$ and $T_y^s$ with $s(x,y)=2$.

\subsection{Geometric Intersection number is 2 and power is 1}

Let us now consider two curves $x,y$ in the torus with geometric intersection number $2$.
By change of coordinates we may assume that the homology classes of the curves are: $\vec x= \spalignmat{1; 0},  \vec y= \spalignmat{1; 2}$. Let $G$ be the subgroup generated by $T_x$ and $T_y$.

\begin{claim} \label{clA}
With notation as above, 
$$ H^1_G=\left\{\spalignmat{a; b}\in \mathbb{Z}^2| \spalignmat{a; b} \text{ is primitive, $b$ is even}\right\}.$$ Moreover, if $z$ is any simple closed curve so that $\vec z\notin H^1_G$, and we let $K$ be the group generated by $T_x, T_y$ and $T_z$, then $K$ is all of Mod($\mathbb{T}^2$); or equivalently  $H^1_K$ is the collection of primitive vectors in $\mathbb{Z}^2.$

\end{claim}

\begin{proof}

We begin by showing the that $ H^1_G\supseteq\left\{\spalignmat{a; b}\in \mathbb{Z}^2| \spalignmat{a; b} \text{ is primitive, $b$ is even}\right\}.$  We will use strong induction on half the absolute value $\frac{|j|}{2}$ of the second coordinate $j$.

Base case $j=0$: Note that the only primitive elements are $\vec x= \spalignmat{1; 0}$ and $-\vec x= \spalignmat{-1; 0}$ which are in $H^1_G$ by definition.\\
Base case $|j|=2$: For any $k\in \mathbb{Z}$, we have 
$T_x^k(\vec y)=\vec y+k\langle \vec x, \vec y\rangle\vec x= \spalignmat{0; 2}+ 2k\spalignmat{1; 0}=\spalignmat{2k+1; 2} $;\\
For any $k\in \mathbb{Z}$, we have 
$T_x^k(-\vec y)=-y+k\langle \vec x,-\vec y\rangle\vec x= \spalignmat{0; 2}+ 2k\spalignmat{1; 0}=-\spalignmat{2k+1; 2} .$ Thus we see all primitive elements with second coordinate $\pm 2$ is in $H^1_G$.\\
Suppose we assume the inductive hypothesis is true for $\frac{|j|}{2}\leq k$. 
Let us consider $\vec v=\spalignmat{a; k+2}$, note that by the division algorithm, we can choose $l\in \mathbb{Z}$ so that the first coordinate $b$ of $$\vec w:=T_x^{l} \vec v=\vec v+\langle x,v\rangle\vec x=\spalignmat{a; k+2}+(k+2)l\spalignmat{1; 0}= \spalignmat{a+(k+2)l; k+2}$$ is satisfies $0< b<k+2$ ( the case $b=0$ does not occur unless $k=\pm 1$, in which case it is a base case), or equivalently $-k-2<k+2-2b<k+2$.
Then we observe that 
$$T_y^{\pm 1} \vec w=\vec w\pm\langle y,w\rangle\vec y=\spalignmat{b; k+2}\pm(k+2-2b)\spalignmat{1; 2}$$

and by choosing $\pm$ sign appropriately, we can ensure that second coordinate is strictly smaller than 
$k+2-2(k+2) =-(k+2)$.

Now we have seen applying the (inverse) Dehn twist about $y$ to $\vec w$ makes the absolute value of second coordinate at most $|k|$, and thus by induction hypothesis we have $T_y^{-1}(T_x^l \vec v)\in H^1_G$, and consequently $\vec v\in H^1_G$.
 A similar calculation (or apply the above to $-\vec v$) holds if we start with $\vec v$ with second coordinate $-(k+2)$. Thus, the inductive step is shown, and this concludes the proof of the inclusion in one direction. 
 The proof of the other inclusion will be deferred until after we prove second half of the proposition.\\

Let us consider a third simple closed curve $z$,  with

$\vec z\notin \left\{\spalignmat{a; b}\in \mathbb{Z}^2| \spalignmat{a; b} \text{ is primitive, $b$ is even}\right\}$, and we will exhibit a simple closed curve $u$ so that $\vec u\in  \left\{\spalignmat{a; b}\in \mathbb{Z}^2| \spalignmat{a; b} \text{ is primitive, $b$ is even}\right\} $, and $(u,z)=1$, and by Fact \ref{2dt}, it will follow that the subgroup $K$ generated by $T_x, T_y$ and $T_z$ is Mod($\mathbb{T}^2$).

\noindent Suppose $\vec z=\spalignmat{p; q}$,  $q$ has to be odd since otherwise $\vec z\in \left\{\spalignmat{a; b}\in \mathbb{Z}^2| \spalignmat{a; b} \text{ is primitive, $b$ is even}\right\}$. 

By Bezout's lemma, there is a vector $\spalignmat{a; b}$ so that $aq-bp=1$. Also the vector $\spalignmat{a+p; b+q}$ has the same property, i.e. $(a+p)q-(b+q)p=1$. One of the primitive vectors $\spalignmat{a; b}$ or $\spalignmat{a+p; b+q}$ is in $H^1_G$ since $q$ odd. Thus, $K$ is all of Mod($\mathbb{T}^2$), and so $H^1_K$ consists of all primitive vectors in $\mathbb{Z}^2$.

Now, by Fact \ref{2dt} we know that $G$ is isomorphic to $F_2$, and thus we must have $$H_G^1= \left\{\spalignmat{a; b}\in \mathbb{Z}^2| \spalignmat{a; b} \text{ is primitive, $b$ is even}\right\},$$ otherwise $G$ would be Mod($\mathbb{T}^2$).

\end{proof}

Hence we conclude, given three curves on a torus with one pairwise geometric intersection number 2, the subgroup generated by them is either $F_2$ or the entire mapping class group Mod($\mathbb{T}^2)$. Moreover, we have:

\begin{prop}\label{gino2}
Given a subgroup $G$ of Mod($\mathbb{T}^2)$ generated by a collection of Dehn twists, if $G$ contains two Dehn twists $T_x$ and $T_y$ where the geometric intersection number of $x$ and $y$ is 2, then $G$ is either isomorphic to the free group $F_2$ or the entire Mod($\mathbb{T}^2)$.
\end{prop}
\begin{proof}
By change of coordinates, and picking the correct orientations, we may assume that $\vec x= \spalignmat{1; 0},  \vec y= \spalignmat{1; 2}$. Let us consider the subgroup $E$ generated by $T_x$ and $T_y$, and by the above discussion, we know exactly what  $H^1_E$ is. Moreover, we know given any other simple closed curve $z$, either $\vec z$ is already in $H^1_E$, or $z$ has geometric intersection number 1 with some element in $H^1_E$ (in which case the subgroup generated by $T_x$, $T_y$ and $T_z$ is Mod$(\mathbb{T}^2)$). So we can simply add the generators of $G$ one at a time, and the result follows.
\end{proof}

More specifically, if we already have  $\vec x= \spalignmat{1; 0},  \vec y= \spalignmat{1; 2}$, then the subgroup generated by $T_x,T_y,T_{z_1},T_{z_2},...$ is $F_2$ if each $\vec z_i$ has second coordinate even, and Mod$(\mathbb{T}^2)$ otherwise.

\subsection{Geometric Intersection number is 1 and power is 2}
The discussions of this Subsection will be similar to the previous one. Let us now consider two curves $x,y$ in the torus with geometric intersection number $1$.
By change of coordinates we may assume that the homology classes of the curves are: $\vec x= \spalignmat{1; 0},  \vec y= \spalignmat{0; 1}$. Let $G$ be the subgroup generated by $T_x^2$ and $T_y^2$. 

\begin{claim}\label{clB}
With notation as above, $$H^2_G=\left\{\spalignmat{a; b}\in \mathbb{Z}^2\mid \spalignmat{a; b} \text{ is primitive, $a+b$ is odd}\right\}.$$ Moreover, if $w$ is any simple closed curve so that $\vec w\notin H^2_G$, and we let $K$ be the group generated by $T_x^2, T_y^2$ and $T_w^2$, then $K$ is isomorphic to $F_2\times C_2$ and  $H^2_K$ is the collection of primitive vectors in $\mathbb{Z}^2.$

\end{claim}

\begin{proof}

We begin by showing the that $ H^2_G\supseteq\left\{\spalignmat{a; b}\in \mathbb{Z}^2| \spalignmat{a; b} \text{ is primitive,  $a+b$ is odd}\right\}.$
We will use strong induction on $\max\{|a|,|b|\}$, where $a,b$ denote the coordinates of the primitive vector.\\
Base case $\max\{|a|,|b|\}=1$: the only possibilities are $a=\pm 1$ and $b=0$;
and $a=0$ and $b=\pm 1$. But these vectors are just $\pm \vec x$ and  $\pm \vec y$, which are in $H^2_G$ by definition.

Let us assume the inductive hypothesis is true for $\max\{|a|,|b|\}\leq n$.
Suppose $\spalignmat{c; d}$ is in $\left\{\spalignmat{a; b}\in \mathbb{Z}^2| \spalignmat{a; b} \text{ is primitive, $a+b$ is odd}\right\}$ with $\max\{|c|,|d|\}=n+1$.
We have $|c|\neq|d|$ , since otherwise $c+d$ is not odd.
Without loss of generality we may assume $|c|>|d|$. Note that 
$ T_x^{2k}\spalignmat{c; d}= \spalignmat{c+2kd; d}$.
By the division algorithm (we may assume the remainder $r$ when divided by $2d$ has the form $-d<r\leq d$; but we know that $r\neq d$ since $c$ and $d$ have opposite parity), we may choose $k\in \mathbb{Z}$ so that $|c+2kd|<|d|$. 
By induction hypothesis
we see that $\spalignmat{c+2kd; d}\in H^2_G$, and consequently it follows that $\spalignmat{c; d}\in H^2_G$. We have shown one inclusion, the other inclusion will follow once again from the proof of the second part of the statement.

Given any primitive vector $\spalignmat{a; b}$ in $\mathbb{Z}^2$, we can similarly use another form of Euclidean algorithm to apply powers of $T_x^2$ and $T_y^2$ repeatedly to decrease maximum of the absolute value of the coordinates, as far as possible. If $a+b$ is odd, we will end up with $\pm \vec x$ or $\pm \vec y$. If $a+b$ is even, we will end up with  $\spalignmat{\pm 1; \pm 1}$. By the discussion just before Claim \ref{csquares}, we may assume that we end up with $\vec z=\spalignmat{1; 1}$. 

For any simple closed curve $w$, if the sum of coordinates of $\vec w$ is even, then $\vec z\in H_K^2$ and consequently $H^2_{K}$ consists of all primitive vectors of $\mathbb{Z}^2$ (since starting with any primitive $\spalignmat{a; b}$ with $a+b$ even, we can always and up at $\vec z$, which means $\spalignmat{a; b}$ is contained in $H^2_{K}$). In particular it follows that the subgroup generated by all squares of Dehn twists is the same as subgroup generated by $T_x^2, T_y^2$ and $T_z^2$; which by Claim \ref{csquares} is isomorphic to $F_2\times C_2$.

Now, by Fact \ref{2dt} we know that $G$ is isomorphic to $F_2$, and thus we must have $$H_G^2= \left\{\spalignmat{a; b}\in \mathbb{Z}^2\mid \spalignmat{a; b} \text{ is primitive, $a+b$ is odd}\right\},$$ otherwise $G$ would be isomorphic to $F_2\times C_2$.
\end{proof}

\begin{remark} Let us note the following consequences of the above proof:
\begin{enumerate}
    \item The subgroup of Mod($\mathbb{T}^2)$ generated by all squares of Dehn twists is isomorphic to $F_2\times C_2$.
    \item If we start with the vectors $\vec x= \spalignmat{1; 0},  \vec y= \spalignmat{0; 1}$ and $\vec z= \spalignmat{1; 1}$, then we can obtain every primitive vector in $\mathbb{Z}^2$ by applying some word in $T_x^2, T_y^2$ and $T_z^2$ to exactly one of $\vec x,\vec y$ or $\vec z$. To see the uniqueness, note that any primitive vector in $\mathbb{Z}^2$ is congruent modulo 2 to exactly one of $\vec x,\vec y$ or $\vec z$, and applying any square of a Dehn twist to a primitive vector preserves the congruence class modulo 2.\end{enumerate}

\end{remark}

Similar to Proposition \ref{gino2}, we have (we skip the proof as it is essentially the same proof as that of Proposition \ref{gino2}):
\begin{prop}\label{gino1}
Given a subgroup $G$ of Mod($\mathbb{T}^2)$ generated by a collection of squares of Dehn twists, if $G$ contains $T_x^2$ and $T_y^2$ where the geometric intersection number of $x$ and $y$ is 1, then $G$ is either isomorphic to the free group $F_2$ or to $F_2\times C_2$.
\end{prop}

More specifically, if we already have  $\vec x= \spalignmat{1; 0},  \vec y= \spalignmat{0; 1}$, then the subgroup generated by $T_x^2,T_y^2,T_{z_1}^2,T_{z_2}^2,...$ is $F_2$ if each $\vec z_i$ has sum of coordinates odd, and $F_2\times C_2$ otherwise.

We showed in the proof of Claim \ref{clB} the subgroup generated by squares of all Dehn twists is isomorphic to $F_2\times C_2$. We note the following consequence:

\begin{prop}The subgroup of Mod$(\mathbb{T}^2)$ generated by $\{T_{x_i}^{s_i}\}_{i\in I}$ is a free group as long as all the $s_i$'s are multiples of 4.

\end{prop}

\begin{proof} With $x,y$ and $z$ as in Claim \ref{csquares}, we note that
by the above remark any $T_{x_i}$ is of the form $w_iT_{u_i}w_i^{-1}$, where $u\in \{x,y,z\}$ and $w_i$ is a word in $T_x^2,T_y^2$ and $T_z^2$. Since $T_z^2=\iota T_x^{2} T_y^{2}$, we can assume $w_i$ is a word in $T_x^2$ and $T_y^2$. $w_iT_{u_i}^{s_i}w_i^{-1}$ is clearly a word in $T_x^2$ and $T_y^2$ if $u\in\{x,y\}$. However if $u=z$, while $T_{x_i}$ need not be a word in $T_x^2$ and $T_y^2$, $T_{x_i}^{s_i}$ is a word in $T_x^2$ and $T_y^2$ as long as $4$ divides $s_i$ since $T_z^4=(T_x^{2} T_y^{2})^2$. Thus the entire subgroup generated by $T_{x_1}^{s_1},...,T_{x_n}^{s_n}$ is contained in the subgroup generated by the $T_x^2$ and $T_y^2$. The result follows since we know that the latter subgroup is isomorphic to $F_2$, and any subgroup of a free group is free.
\end{proof}


\section{Ping pong with sliding}

We have seen that given any three powers of Dehn twists, we can apply the Euclidean algorithm and get down to at most three Dehn twist powers so that the corresponding curves have comparable geometric intersection number. Moreover, if we started with uniform powers different from two, we know by Theorem \ref{thmB} and Fact \ref{2dt} what subgroups the generate. In this section, we will describe a procedure by which we can hope to find what subgroup is generated by a collection of Dehn twists\footnote{The procedure may be modified to begin with any Dehn twist powers, although it is not quite clear what the right analogue of terminate would be.}. As we will see later in this section, we are able to figure out what some subgroups are by this procedure, which we were unable to do by directly using the results from previous section. However, it is not clear if this procedure always terminate, which is why we are not calling it an algorithm.

For this procedure we will use in addition to the ping pong lemma (more specifically Theorem \ref{thmA}), sliding moves for Dehn twist factorizations. Sliding moves were originally studied by Hurwitz \cite{Hu}, they come up naturally in studying the action of the braid group $B_n$ on the fundamental group of the $n$ times punctured disc. Typically sliding moves are studied for tuples of elements, we will apply them for a elements of the mapping class group written as a product of Dehn twists. Given a product of (left and right handed) Dehn twists $g_ng_{n-1}...g_2g_1$  
we can pick any element, say $g_i$ and slide it to the very right of the word at the expense of conjugating (by $g_i$) the elements on the left of the chosen $g_i$. To elaborate we can replace $g_{i}g_{i-1}$ with $(g_{i}g_{i-1}g_{i}^{-1})g_{i}$ in the word  (keeping everything else the same), and the element $w$ remains the same. We could also use the inverse sliding move $g_{i+1}g_i\rightarrow g_i(g_i^{-1}g_{i+1}g_i) $ to move the element $g_i$ in the to the left of the word.

\begin{remark}\label{sliding}

 Given any element $w$ in Mod($\mathbb{T}^2$), written as a product of Dehn twist powers $T_x^s$, $T_{y_1}^{s_1},...,T_{y_n}^{s_n}$, with a Dehn twist $T_x$ appearing a factorization of $w$, we can rewrite $w$ as a product of $T_x^{\eta s}$ (where $\eta\in \{0,1\}$), times a word in $T_x^{2s}$, $T_{y_1}^{s_1},...,T_{y_n}^{s_n}$, and $T_{z_1}^{s_1},...,T_{z_n}^{s_n}$ where $z_i=T_x^s(y_i)$.
    
    For example, consider the word $w=T_{y_1}^5T_x^{-3} T_{y_1}^{-2}T_x^{3}T_{y_2}T_x^{-1}T_{y_3}^4$ with all $s_i$ and $s$ equal to 1. Since we only want conjugates of $T_{y_i}$ by $T_x$ (recall $T_{f(x)}=f T_x f^{-1}$ for any homeomorphism $f$), we will rewrite the leftmost odd power of $T_x$ as $T_x$ times an even power of $T_x$ and slide the $T_x$ to the right till we come across another power of $T_x$, whence we repeat the process.
    $$w=T_{y_1}^5T_x^{-3} T_{y_1}^{-2}T_x^{3}T_{y_2}T_x^{-1}T_{y_3}^4= T_{y_1}^5T_x^{-4}T_x T_{y_1}^{-2}T_x^{3}T_{y_2}T_x^{-1} T_{y_3}^4$$
    $$= T_{y_1}^5T_x^{-4} T_{z_1}^{-2}T_xT_x^{3}T_{y_2}T_x^{-1} T_{y_3}^4=T_{y_1}^5T_x^{-4} T_{z_1}^{-2}T_x^{4}T_{y_2}T_x^{-1} T_{y_3}^4 $$
    $$= T_{y_1}^5T_x^{-4} T_{z_1}^{-2}T_x^{4}T_{y_2}T_x^{-2}T_xT_{y_3}^4=T_{y_1}^5T_x^{-4} T_{z_1}^{-2}T_x^{4}T_{y_2}T_x^{-2} T_{z_3}^4T_x.$$

\end{remark}

We now claim that if after applying one of these moves if a collection of Dehn twist powers generate a free subgroup, then the original collection has the same property.
\begin{prop}\label{slidefr}
Suppose we have a collection of distinct Dehn twist powers  $T_x^s$, $T_{y_1}^{s_1},...,T_{y_n}^{s_n}$, and we set $z_i=T_x^s(y_i)$ for all $i$. If the subgroup $\langle  T_x^{2s}, T_{y_1}^{s_1},...,T_{y_n}^{s_n},T_{z_1}^{s_1},...,T_{z_n}^{s_n}\rangle$ of Mod($\mathbb{T}^2$) is freely generated by $\{  T_x^{2s}, T_{y_1}^{s_1},...,T_{y_n}^{s_n},T_{z_1}^{s_1},...,T_{z_n}^{s_n}\}$, then the subgroup $\langle  T_x^{s}, T_{y_1}^{s_1},...,T_{y_n}^{s_n}\rangle$ of Mod($\mathbb{T}^2$) is freely generated by $\{T_x^{s},T_{y_1}^{s_1},...,T_{y_n}^{s_n}\}$. The same conclusion also holds if we instead set $z_i=T_x^{-s}(y_i)$ for all $i$.

\end{prop}

\begin{proof}
Suppose we have a non trivial relation $w=1$ among the Dehn twist powers $\{T_x^{s},T_{y_1}^{s_1},...,T_{y_n}^{s_n}\}$. By sliding (as explained in Remark \ref{sliding}) and if necessary moving the rightmost $T_x^s$ to the other side of the equality, we see that either we have a non-trivial relation among $\{  T_x^{2s}, T_{y_1}^{s_1},...,T_{y_n}^{s_n},T_{z_1}^{s_1},...,T_{z_n}^{s_n}\}$, or $T_x^{-s}$ is a word in $\{  T_x^{2s}, T_{y_1}^{s_1},...,T_{y_n}^{s_n},T_{z_1}^{s_1},...,T_{z_n}^{s_n}\}$. The latter case also gives rise to a similar contradiction, since $T_{z_1}^{s_1}=T_x^sT_{y_1}^{s_1}T_x^{-s}$ (alternately we can use the fact that $T_x^{2s}=(T_x^s)$ is also a non-trivial word in  $\{  T_x^{2s}, T_{y_1}^{s_1},...,T_{y_n}^{s_n},T_{z_1}^{s_1},...,T_{z_n}^{s_n}\}$).

For justification of the last sentence in the proposition, observe that we could instead slide powers of $T_x^s$ to the left, and in this case the conjugate of $T_{y_i}$ we see in the new word is $T_x^{-s}(y_i)$. 
\end{proof}

We are now ready to state our procedure:\\
\noindent\textbf{Procedure}: Suppose we start with a collection of Dehn twists $T_{x_1}^{s_1},...,T_{x_n}^{s_n}$. If the collection of simple closed curves $x_1,...,x_n$ is not $\{s_i\}$-proportional, then we find a triple $x_i, x_j$ and $x_k$ that violates the $\{s_i\}$-proportional inequality, that is it satisfies $$(x_i,x_j)+(x_j,x_k)+(x_k,x_i)> s_j(x_i,x_j)(x_j,x_k).$$
We then replace $T_{x_j}^{s_j}$ with $T_{x_j}^{2s_j}$, keep the rest of the $T_{x_l}^{s_l}$ unchanged, and append the Dehn twist powers  $T_{T_{x_j}^{ s_j}(x_l)}^{s_l}$  for all $l\neq j$ (or append  $T_{T_{x_j}^{ -s_j}(x_l)}^{s_l}$ for all $l\neq j$) to our collection. We keep repeating the above step. If after a finite number of steps we get a collection which satisfies the $\{s_i\}$- proportional inequalities, then we can conclude the original collection of Dehn twist powers are free in those generators.

In case the original collection we started of with is a collection of Dehn twists, we will keep a separate list where we keep track of all Dehn twists that we come across in the above process. Let us say the procedure \emph{terminates} in case we get a pair of Dehn twists where the geometric intersection number of a pair of Dehn twists is at most two, or after some finite number of steps we get a collection which have  with  $\{s_i\}$- proportional geometric intersection numbers.\\

Recall that if geometric intersection number of two simple closed curves $x$ and $y$ is 1, then $T_x$ and $T_y$ generate Mod($\mathbb{T}^2)$ (see Fact \ref{2dt}) and if the geometric intersection number is 1, then the collection of Dehn twists can generate either $F_2$ or Mod($\mathbb{T}^2)$ (see Proposition \ref{gino2}). Thus if the procedure starting with a collection of Dehn twists terminates, then we can figure out precisely what subgroup they generate. This brings up the natural question:
\begin{question}
Starting with a collection of Dehn twists, does the procedure mentioned above always terminate? If the answer is "No", can one find an algorithm (or procedure) which tells us the subgroup generated by given a collection of Dehn twists in the torus?
\end{question}

Let us now apply the procedure to Example \ref{tricky} and figure out the subgroup generated.
\begin{example}\label{eg7} We start with the $\vec x= \spalignmat{1; 0},\vec y= \spalignmat{1; 3},\vec z= \spalignmat{1; 10},\vec w= \spalignmat{3; 17}$. 
We see that when we apply the procedure and replace $T_y$ with $T_y^2$, we need to append the vectors:
$$T_y^{-1}\vec x=\spalignmat{4;9}, T_y^{-1}\vec z=-\spalignmat{6;11}, \text{ and } T_y^{-1}\vec w=-\spalignmat{5;7}.$$
and it can be checked that they satisfy the $\{s_i\}$- proportional inequalities. Thus the subgroup is generated by Dehn twists about given collection of curves is isomorphic to $F_4$, freely generated by $T_x,T_y,T_z$ and $T_w$. 
\end{example}

Let us now discuss an example with higher powers of Dehn twists.
\begin{example}\label{eg8}
Consider the subgroup generated by $T_x^4, T_{y_0}^4, T_{y_1}^4, T_{y_2}^4,$ and $T_{y_3}^4$, where $\vec x=\spalignmat{1;0}$, and $\vec y_i=\spalignmat{i;1}$ for $0\leq i\leq 3$. This time we see that the given collection of vectors do not have 4-proportional geometric intersection numbers. The procedure above requires us to replace $T_x^4$ with its square $T_x^8$, and append four more vectors $T_{y_4}^4, T_{y_5}^4, T_{y_6}^4,$ and $T_{y_7}^4$ (using the notation $y_i=\spalignmat{i;1}$ with $i$ an arbitrary integer). We still see this new collection do not have proportional geometric intersection numbers, precisely because of $\vec x$ and the two extreme $\vec y_i$. This pattern continues no matter how many times we raise the power of $T_x^4$. While the above procedure fails, it turns out the subgroup generated by the original collection of fourth powers is $F_5$, freely generated by the aforementioned collection. We can justify this by sliding all powers of $T_x$, we will obtaining the above statement as a special case of Proposition \ref{free} below, after we discuss a more general setup.


\end{example}

We begin by stating an analogue of Remark \ref{sliding}, where we slide all powers of $T_x$ to one side.
\begin{remark}
Given any element $w$ in Mod($\mathbb{T}^2$), written as a product of Dehn twists $T_x$, $T_{y_1},...,T_{y_n}$, with a Dehn twist $T_x$ appearing a factorization of $w$, we can rewrite $w$ as a product of $T_x^\eta$ (for some integer $\eta$), times a word in  $T_{y_1},...,T_{y_n}$ and all possible conjugates of $T_{y_1},...,T_{y_n}$ by various powers of $T_x$.
   
    For example, let us consider the same word as above $w=T_{y_1}^5T_x^{-3}T_{y_1}^{-2}T_x^{3}T_{y_2}T_x^{-1}T_{y_3}^4$. This time we will bring all powers of $T_x$ to the right of $w$ by sliding.
     $$w=T_{y_1}^5T_x^{-3}T_{y_1}^{-2}T_x^{3}T_{y_2}T_x^{-1}T_{y_3}^4=
    T_{y_1}^5 T_{T_x^{-3}(y_1)}^{-2}T_x^{-3}T_x^{3}T_{y_2}T_x^{-1}T_{y_3}^4$$
    $$=T_{y_1}^5 T_{T_x^{-3}(y_1)}^{-2}T_{y_2}T_x^{-1}T_{y_3}^4=T_{y_1}^5 T_{T_x^{-3}(y_1)}^{-2}T_{y_2}T_{T_x^{-1}(y_3)}^4T_x^{-1}.$$
    
\end{remark}

By a very similar argument as in Proposition \ref{slidefr} we have: 

\begin{prop}\label{slidefull}
Suppose we have a collection of distinct Dehn twist powers  $T_x^s$, $T_{y_1}^{s_1},...,T_{y_n}^{s_n}$. If the subgroup $\langle   T_{T_x^{sp_1}(y_1)}^{s_1},...,T_{T_x^{sp_n}(y_n)}^{s_n}|p_i\in\mathbb{Z}\rangle$ of Mod($\mathbb{T}^2$) is freely generated by $\{ T_{T_x^{sp_1}(y_1)}^{s_1},...,T_{T_x^{sp_n}(y_n)}^{s_n}|p_i\in\mathbb{Z}\}$, then the subgroup $\langle  T_x^{s}, T_{y_1}^{s_1},...,T_{y_n}^{s_n}\rangle$ of Mod($\mathbb{T}^2$) is freely generated by $\{T_x^{s},T_{y_1}^{s_1},...,T_{y_n}^{s_n}\}$.
\end{prop}

 \begin{prop}\label{free} Let $s\geq 4$, consider the simple closed curves corresponding to $\vec x=\spalignmat{1; 0}$, and for $0\leq k \leq s-1$ the vectors $\vec y_k=\spalignmat{k; 1}$.
 Then the collection of $s$-th powers of Dehn twists along the aforementioned simple closed curves freely generates a subgroup of Mod$(\mathbb{T}^2)$. The same conclusion would hold if we let $k$ vary over any collection of representatives of congruence classes modulo $s$.
 \end{prop}

\begin{proof}
By the discussion just prior to the proposition, we see that we can replace the given collection with $T^s_{y_l}$ for all $l\in\mathbb{Z}$ (observe that this happens as long as we start with any any collection of representatives of congruence classes modulo $s$). By Example \ref{eg}, we know that $\{ T^s_{y_l}\}_{l\in \mathbb{Z}}$ freely generate a subgroup, and thus the result follows by Proposition \ref{slidefull}.
\end{proof}

The above discussion suggests that in some situations, one can modify the procedure by sliding in different ways and figure out what subgroup is generated by a collection of Dehn twist powers.


\section{Subgroups generated by uniform powers of all Dehn twists in the torus}

In this section we discuss the structure of $N_s$, subgroup of generated by $s$-th powers of all Dehn twists in the torus. We already know  for $s=1$ the answer is Mod($\mathbb{T}^2)$ and for $s=2$ the answer is $F_2\times C_2$. We observe that $N_s$ is a normal subgroup of Mod($\mathbb{T}^2)$, as conjugates of Dehn twists are Dehn twists. Also since $$\spalignmat{1 1; 0 1}^s=\spalignmat{1 s; 0 1}\equiv\spalignmat{1 0; 0 1}\pmod{s}$$ $N_s$ is a subgroup of the principal congruence subgroup $\Gamma_s$ of level $s$ (the set of matrices in $SL_2(\mathbb{Z})$ which are congruent to the identity modulo $s),$

For $s\geq 3$, the subgroups $\Gamma_s$ are known to be free by any of the following methods:
\begin{itemize}
    \item It can be checked that for $s\geq 3$, $\Gamma_s$ is torsion free and maps isomorphically\footnote{this is equivalent to stating
    that we cannot have a matrix $A$ so that both $A$ and $-A$ are in $\Gamma_s$ for $s>2$.} to its image in $PSL_2(\mathbb{Z})$.
    Kurosh subgroup theorem for $PSL_2(\mathbb{Z})\cong \mathbb{Z}_2*\mathbb{Z}_3$ implies any subgroup is isomorphic to a free product of $\mathbb{Z}_2$ and $\mathbb{Z}_3$; and such groups 
    are free if and only if they are torsion free.
    \item Group action on trees \cite[Chapter 2]{OH}.
    \item Hyperbolic geometry and covering space theory, see for instance the second and third paragraphs of the second proof of Theorem \ref{thmE}.
\end{itemize}
Since subgroups of free groups are free, it follows that $N_s$ are all free for $s\geq 3$. We can say more about the structure of $N_s$.

\begin{thm} \label{thmE}

 $N_s$ is a subgroup of Mod($\mathbb{T}^2)$ of index $1,6,24,48,120$ and infinite, isomorphic to  $SL_2(\mathbb{Z})$, $F_2\times C_2$, $F_3$, $F_5$, $F_7$ and $F_\infty$ for $s$ equal to $2,3,4,5$ and bigger than 5 respectively. In particular, $N_s$ coincides with the principal congruence subgroup $\Gamma_s$ for $s\leq 5$, and is an infinite index subgroup of $\Gamma_s$ for $s>5$.
 Moreover for $2\leq s \leq 5$, the following primitive vectors correspond to the simple closed curves, so that $s$-th powers of Dehn twists about all but one of them freely generate  $F_s$ for $s=3,4,5$ respectively:
 \begin{itemize}
     \item $ \spalignmat{1; 0}$,  $ \spalignmat{0; 1}$,  $\spalignmat{1; 1}$, $\spalignmat{1; -1}$.
     \item $ \spalignmat{1; 0}$,  $ \spalignmat{0; 1}$,  $\spalignmat{1; 1}$, $\spalignmat{1; -1}$, $\spalignmat{1; 2}$, $\spalignmat{2; 1}$.
     \item $ \spalignmat{1; 0}$,  $ \spalignmat{0; 1}$,  $\spalignmat{1; 1}$, $\spalignmat{1; -1}$, $\spalignmat{1; 2}$, $\spalignmat{2; 1}$, $\spalignmat{1; -2}$, $\spalignmat{-2; 1}$, $\spalignmat{2; 3}$, $\spalignmat{-2; 3}$, $\spalignmat{2; 5}$, $\spalignmat{5; 2}$.
 \end{itemize}
\end{thm}

This result above is mostly available in the literature, and we will reference the relevant results in our first proof. We will give another proof of this result using covering space theory, which is probably well known, but to the best of the authors knowledge does not appear in the literature.

 \begin{proof}  That the indices of the subgroup $N_s$ are as in the theorem is stated in \cite{H2} (see paragraph after Theorem 4 there), based on results in  \cite{N}. It is easy to see that for $s\leq 5$, these are also the indices of $\Gamma_s$ in Mod($\mathbb{T}^2)=SL_2(\mathbb{Z})$, or equivalently the cardinality of $SL_2(\mathbb{Z}_s)$.  Since we know that for $s>5$, the subgroups $N_s$ are infinite index in $\Gamma_s$ (as $\Gamma_s$ has finite index in $SL_2(\mathbb{Z}_s)$), by a theorem of Schreier \cite{Sc} (see also \cite{KS}) $N_s$ cannot be finitely generated since it is a normal subgroup.
  By work of Kulkarni\cite{Ku}, one can find free generators of $\Gamma_s$, and this has been implemented \cite{KL} in computer algebra systems\footnote{ For  instance, in SageMath one can find the free generators of $\Gamma_s$ using the command "Gamma($s$).generators  ()" for a specific value of $s$.}.  By direct computations it can be verified that the Dehn twists about the curves stated in the theorem generate $\Gamma_s$ freely for $3\leq s \leq 5$.

 \end{proof}

\noindent Most of the following proof was outlined to us by Dan Margalit.

\begin{proof} [Proof of Theorem \ref{thmE} using covering space theory]
By covering space theory, we know that for any subgroup of $G$ of $PSL_2(\mathbb{Z})$ (which clearly is a discrete subgroup of Isom$^+(\mathbb{H}^2)\cong PSL_2(\mathbb{R})$), the space $\mathbb{H}^2/G$ has fundamental group $G$. We would like to understand the quotient when $G$ is $\Gamma_s$ and $N_s$.
For doing this it is easier to think of the quotient of Farey complex (see \cite[Chapter 2]{OH}), which is an ideal triangulation of $\mathbb{H}^2$. Once we figure out the quotient of the Farey complex by $G$, we can simply remove the ideal vertices, and obtain the quotient of the $\mathbb{H}^2$ by $G$. Recall that the vertices of the Farey complex $F$ are the projectivised\footnote{This means we identify each primitive vector $\vec v$ with its negative $-\vec v$.} primitive vectors in $\mathbb{Z}^2$, there is an edge between two vertices $\pm \vec v$ and $\pm \vec w$ if and only if their integer span is all of $\mathbb{Z}^2$ (or, equivalently the determinant of the matrix with columns $\pm \vec v$ and $\pm \vec w$ 
is $\pm 1$), and all triangles are filled in.

The quotient of the Farey complex by $\Gamma_s$ is a triangulated complex and has the following similar description (essentially this boils down to taking the congruence classes) as follows:
vertices are  projectivised primitive\footnote{Here primitive means that the vector is not a multiple of another vector by a non-unit, or equivalently the ideal generated by the two coordinates is the unit ideal. If $s$ is a prime, then any non-zero vector in $\mathbb{Z}_s^2$ is primitive.} vectors in $\mathbb{Z}_s^2$, there is an edge between two vertices if and only if their integer span is all of $\mathbb{Z}_s^2$, and all triangles are filled in. Note that in $\sfrac{F}{\Gamma_s}$ the  vertex $ \pm\spalignmat{ 1; 0}$ in  has exactly $s$ vertices adjacent to it $$\pm\spalignmat{0;  1},\pm\spalignmat{1;  1}, ..., \pm\spalignmat{s-1;  1} $$ 
and the same is true for any other vertex
in $\sfrac{F}{\Gamma_s}$. It follows that each vertex of the quotient $\sfrac{F}{\Gamma_s}$ has degree $s$, and moreover each edge of $\sfrac{F}{\Gamma_s}$ is adjacent to exactly two triangles. If we denote the number of vertices, edges and faces of the quotient  $\sfrac{F}{\Gamma_s}$ by $v,e,f$ respectively, we see that $2e=sv$ and $3f=2e$. Hence the Euler characteristic of $\sfrac{F}{\Gamma_s}$ is $$v-e+f=v-e+\frac{2}{3}e=v-\frac{1}{3}e=v-\frac{s}{6}v=\frac{v(6-s)}{6}$$

Thus we see that this is a closed surface of positive Euler characteristic (i.e. spherical)  if $s<6$,
zero Euler characteristic (i.e. flat) if $s=6$, and negative Euler characteristic (i.e. hyperbolic) if $s>6$. We also observe that since $\Gamma_s$ consists of orientation preserving isometries of $\mathbb{H}^2$, the surface $\sfrac{F}{\Gamma_s}$ must be a closed oriented formula, and we can find its genus by the formula $g=1-\frac{\chi}{2}$.

As we discussed earlier, $\sfrac{\mathbb{H}^2}{\Gamma_s}$ is the surface $\sfrac{F}{\Gamma_s}$ punctured at $v$ vertices. Since none of them are closed surfaces (we are throughout assuming $s\geq 3$ in this discussion), we see that the fundamental group $\Gamma_s$ must be free of finite rank, with one generator for each puncture (i.e. number of vertices in  $\sfrac{F}{\Gamma_s}$) and one generator for each one-handle. This gives an alternative proof of the freeness of $\Gamma_s$ for $s\geq 3$, as mentioned earlier. 

Let us now try to understand the quotients $\sfrac{\mathbb{H}^2}{N_s}$. First we make the observation the parabolic isometry $x\mapsto x+s$ (in the upper half plane model of $\mathbb{H}^2$) corresponds\footnote{ Distinct non-trivial  powers of Dehn twists are conjugate in $PSL_2(\mathbb{R})$, but not in $PSL_2(\mathbb{Z})$.} to the matrix $\spalignmat{1 s; 0 1}$ in $PSL_2(\mathbb{Z})$, i.e. the $s$-th power of the Dehn twist about $\spalignmat{1;0}$. The fundamental group of $\sfrac{\mathbb{H}^2}{N_s}$ is by definition generated by the various loops in the surface $\sfrac{\mathbb{H}^2}{N_s}$. The peripheral loops (the ones surrounding punctures formed by removing the images of the ideal vertices from $\sfrac{F}{N_s}$) correspond to parabolic elements in the hyperbolic isometry group, i.e. some power of a Dehn twist.

Other loops correspond to hyperbolic\footnote{ No loops correspond to elliptic isometries since their action is not free.} elements in the isometry group. Since $N_s$ is generated by $s$-th powers of all Dehn twists, it follows that if we fill in all the punctures, all those peripheral loops now are trivial. Thus $\sfrac{F}{N_s}$ is simply connected, and so $\sfrac{\mathbb{H}^2}{N_s}$ must in fact be a sphere with (possibly infinitely many) punctures for any $s\geq 3$. Since $N_s$ is a subgroup of $\Gamma_s$, we get a covering space $\sfrac{\mathbb{H}^2}{N_s}\rightarrow \sfrac{\mathbb{H}^2}{\Gamma_s}$, and the number of sheets in the covering equals the index of the subgroup $N_s$ in $\Gamma_s$.

For $s\geq 6$, we know that $\sfrac{\mathbb{H}^2}{\Gamma_s}$ has positive genus, and any finite sheeted cover will also have positive genus. This means that for $s\geq 6$ this covering must be infinite sheeted as we saw earlier that $\sfrac{\mathbb{H}^2}{N_s}$ has zero genus. Moreover since covering space $\sfrac{\mathbb{H}^2}{N_s}\rightarrow \sfrac{\mathbb{H}^2}{\Gamma_s}$ is a normal, $N_s$ has to be a free group of countably infinite rank.


We observe that for $3\leq s\leq 5$, the collection of representative classes of primitive vectors in $\mathbb{Z}_s^2$ are precisely the ones listed in the last line of the statement of Theorem \ref{thmE}. Moreover, for $s=3,4,5$ the quotient $\sfrac{F}{\Gamma_s}$ is a regular triangulation of the sphere with exactly four, six and twelve vertices (i.e. a tetrahedron, octahedron and icosahedron) respectively. Thus, $\sfrac{\mathbb{H}^2}{\Gamma_s}$ is a sphere with exactly four, six and twelve puctures respectively for $s=3,4,5$. We know that the fundamental group of a sphere with $n$ punctures has the presentation $\langle x_1,...,x_n|x_1....x_n\rangle$ where the $x_i$ is the loop surrounding the $i$-th puncture. Hence the fundamental group of $\sfrac{\mathbb{H}^2}{\Gamma_s}$, $\Gamma_s$ is generated freely by all but one of the peripheral loops around the punctures, and as we discussed earlier, these are precisely the $s$-th powers of Dehn twists about these curves. So we conclude that $\Gamma_s$ coincides with $N_s$ for $s=3,4,5$, and we also obtain the statement about free generation by the given collection of Dehn twists.

\end{proof}

\begin{remark}
We can also show that for $s=3,4$, the subgroups $N_s$ are freely generated by all but one of the $s$-th powers of Dehn twists about the collection of curves stated in Theorem \ref{thmE} by using Theorem \ref{thmA}, as we explain below. One first uses induction to show that (this is similar to Claims
\ref{clA} and \ref{clB}) if we consider the subgroup $G$ generated by $s$-th powers of Dehn twists about the given curves, the subset $H^s_G$ consists of all primitive vectors in $\mathbb{Z}^2$. 

For $s=3$,  we set $\vec x= \spalignmat{1; 0}$,  $\vec y= \spalignmat{0; 1}$,  $\vec z= \spalignmat{1; 1}$ and $\vec w=\spalignmat{1; -1}$. We see that 

$$\spalignmat{1; -1} \xrightarrow{T_x^3} \spalignmat{-2; -1} \xrightarrow{T_z^3} \spalignmat{1; 2}  \xrightarrow{T_y^3}\spalignmat {1; -1} ,$$
consequently $T_y^3T_z^3T_x^3$ must be some power of $T_w$, and it turns out that $T_y^3T_z^3T_x^3=T_w^{-3}$. Since $x,y$ and $z$ are 3-proportional, the statement follows by Theorem $\ref{thmA}$.
\end{remark}

For $s=4$, let us define $\vec u=\spalignmat{1; 2}$, $ \vec v=\spalignmat{2; 1}$,  $\vec w=\spalignmat{1; -1}$, $\vec x= \spalignmat{1; 0}$, $\vec y= \spalignmat{0; 1}$, $\vec z= \spalignmat{1; 1}$. Now we observe that 

$$\spalignmat{1; 2} \xrightarrow{T_y^4} \spalignmat{1; -2} \xrightarrow{T_w^4} \spalignmat{-3; 2}  \xrightarrow{T_{x}^4}\spalignmat {5; 2}\xrightarrow{T_{v}^4}\spalignmat {-3; -2} \xrightarrow{T_{z}^4}\spalignmat {1; 2}.$$
Thus $T_z^4T_v^4T_{x}^4T_w^4T_y^4$ has to be some power of $T_u$, and in fact turns out to equal $T_u^{-4}$. Thus $N_4$ is the subgroup generated by $T_v^4$, $T_w^4, T_x^4, T_y^4$ and $T_z^4$, which we know is free by Proposition \ref{free}.




\section{Subgroups of mapping class group of one holed torus and the braid group on three strands}

In this section we discuss implications of the above results to subgroups generated by half-twists in the braid group $B_3$ on three strands, or subgroups generated by Dehn twists in the one holed torus $S^1_{1}$ (i.e. a torus with one boundary component).

Recall that the braid group $B_3$ is isomorphic to the mapping class group Mod$(D,3)$ of the  two dimensional disk with three marked points. The double branched cover of $D$ with three branched points is the one holed torus, and half twists about in Mod$(D,3)$ correspond to Dehn twists about lift of that arc in Mod$(S_1^1)$, and this induces an isomorphism $$ \text{Mod}(S_1^1)\cong \text{Mod($D$,3)}\cong \langle a,b|aba=bab\rangle, $$where $a$ and $b$ can be identified with the standard Artin generators $\sigma_1$ and $\sigma_2$ in $B_3$, or to the Dehn twists about the longitude and meridian in $S_1^1$. There is an analogous, but more complicated statement for braid groups with more strands, see \cite{BH}.
The mapping class group of the torus is obtained by adding to the presentation of Mod$(S_1^1)$ the relation $(ab)^6=1$ corresponding to capping of the boundary. So if we have a sequence of essential simple closed curves in $S_1^1$, and a relation among Dehn twists about them (or equivalently relations among half twists in $B_3$), we would still have a relation in Mod$(\mathbb{T}^2)$ once we cap off the boundary. 
In all of our previous discussion about subgroups in the mapping class group of the torus, the only cases where the subgroup was not a free group was Mod$(\mathbb{T}^2)$ (when we had two Dehn twists whose corresponding curves had geometric intersection number one) and $F_2\times C_2$ (which was generated by squares of Dehn twists about the $\spalignmat{1;0},\spalignmat{0;1} $ and $\spalignmat{1;1}$ curves). Let us now figure out what subgroups these Dehn twists generate in Mod$(S_1^1)$ (note that there is an isomorphism of first homology $H_1(S_1^1)\cong H_1(\mathbb{T}^2)$ and there is an analogous correspondence between simple closed curves in $S_1^1$ and primitive homology classes, up to orientation).

For the first case, we see that it is isomorphic to the subgroup generated by $a$ and $b$, which is Mod$(S_1^1)\cong B_3$. For the subgroup generated by the squares of Dehn twists about the curves $\vec x=\spalignmat{1;0},\vec y=\spalignmat{0;1} $ and $\vec z=\spalignmat{1;1}$, we note that $T_x^2=a^2$, $T_y^2=b^2$ and $T_z^2=b^{-1}a^2b$ and thus $$T_x^2T_y^2T_z^2=a^2b^2b^{-1}a^2b=a(aba)ab=ababab=\Delta^2,$$
the square of the Garside element, which generates the center of $B_3$. It follows that the subgroup generated by  $T_x^2$, $T_y^2$ and $T_z^2$ is isomorphic to the subgroup generated by $T_x^2$, $T_y^2$ and $\Delta^2$. This subgroup is isomorphic to $F_2\times F_1$ (the internal direct product of subgroup generated by $T_x^2$, $T_y^2$; and the infinite cyclic subgroup generated by  $\Delta^2$). 

Thus we can translate our earlier results in the torus and obtain similar results in  
Mod$(S_1^1)=B_3$. For instance, here is the analogue of Theorem \ref{thmD}.

\begin{thm}
The subgroup of $B_3$ generated by $s$-th powers of all half twists (or equivalently, the subgroup of Mod$(S_1^1)$ generated by $s$-th powers of all essential Dehn twists) is isomorphic to  $B_3$, $F_2\times F_1$, $F_3$, $F_5$,  $F_{11}$ and $F_\infty$ for $s$ equal to one, two, three, four, five, and greater than five  respectively.
\end{thm}

We remark that it is important that we talk about Dehn twists about essential simple closed curves in $S_1^1$, since if we allowed Dehn twist about boundary parallel curve $c$, there may be more relations as $T_c$ is a central element. If the subgroup $G$ of Mod$(S_1^1)$ does not contain any central element, then the subgroup of Mod$(S_1^1)$ generated by $G$ and $T_c^p$ (where $p$ is non-zero) is isomorphic to $G\times F_1$.  The situation is more complicated if $G$ contains some central element. 

 If we allow Dehn twists about the boundary parallel curve, the analogue of the previous theorem becomes:

\begin{thm}
 The subgroup of Mod$(S_1^1)$ generated by $s$-th powers of all Dehn twists (including Dehn twist about the boundary parallel curve $c$) is isomorphic to  $B_3$, $F_2\times F_1$, $F_3\times F_1$ and $F_5\times F_1$,  $F_{11}\times F_1$ and $F_\infty\times F_1$ for $s$ equal to one, two, three, four, five, and greater than five  respectively.
\end{thm}

\begin{proof}
We observe that the subgroup generated by $s$-th powers of all essential Dehn twists contains $T_c=\Delta^4$ when $s$ equals 1 or 2, and does not contain any central element when $s$ is at least 3 (otherwise those subgroups will not be free).
\end{proof}


\begin{thebibliography}{10}

 \bibitem{BM} S Bachmuth, and H Mochizuki.
 \newblock  Triples of 2×2 matrices which generate free groups.
 \newblock {\em Proc. Am. Math Soc. }, 59: 25-28., 1976. 
 
\bibitem{Br} J. L. Brenner. 
 \newblock  Queiques groupes libres de matrices.
 \newblock {\em C. R. Acad. Sci. Paris}, 24: 1689–1691., 1955. 
 
\bibitem{BH}  J. S. Birman and H. M. Hilden.
\newblock On the mapping class groups of closed surfaces as covering spaces.  
 \newblock {\em Ann. of Math. Studies}, No. 66,81–115 ,1971.

\bibitem{CJR}  B.  Chang,  S. A. Jennings, and  R. Ree.
 \newblock  On certain matrices which generate free groups.
 \newblock {\em Can. J. Math.}, 10: 279–284, 1958. 

\bibitem{OH} Matt Clay and Dan Margalit (editors).
\newblock Office Hours with a Geometric Group Theorist.
 \newblock {\em Princeton University Press}, 2018.
 
 \bibitem{D}
 Max Dehn.  
 \newblock The group of mapping classes.
 \newblock {\em In Papers on group theory and topology, Springer-Verlag, New York,} 1987. Translated from the German and with introductions and an appendix by John Stillwell. 
 
  \bibitem{FM}
 Benson Farb and Dan Margalit.
 \newblock  Primer on Mapping Class Groups.
 \newblock {\em Princeton University Press,} 2011.
 
 \bibitem{T} A. Fathi, F. Laudenbach, and V. Po\'{e}naru, editors.
 \newblock Travaux de Thurston sur les surfaces.
 \newblock {\em  Ast\'{e}risque volume 66}, Soci\'{e}t\'{e} Math\'{e}ematique de France, Paris, 1979.
S\'{e}minaire Orsay, with an English summary.
 

 \bibitem{HT} Hessam Hamidi-Tehrani.
\newblock Groups generated by positive multi-twists and the fake lantern problem.
 \newblock {\em Algebr. Geom. Topol.}, 2:1155–1178, 2002.
 
 
  \bibitem{H0}   Stephen P. Humphries. 
 \newblock Generators for the mapping class group.
 \newblock {\em Topology of low dimensional manifolds
(Chelwood Gate, 1977)}, volume 722 of Lecture Notes in Mathematics:44–47, 1979.

 \bibitem{H1}  Stephen P. Humphries. 
 \newblock Free products in mapping class groups generated by Dehn twists.
 \newblock {\em Glasgow Mathematical Journal}, 31(02):213-218, 1989.
 
\bibitem{H2}  Stephen P. Humphries. 
 \newblock Normal closures of powers of Dehn twists in mapping class groups.
 \newblock {\em   Glasgow Mathematical Journal}, 34(3):314-317, 1992. 



 \bibitem{Hu}  A. Hurwitz.
\newblock Ueber Riemann'sche Fl\"{a}chen mit gegebenen Verzweigungspunkten.
\newblock {\em  Math. Ann.}, 39: 1-60, 1891.

 
\bibitem{I}  Atsushi Ishida.
\newblock The structure of subgroup of mapping class groups \newblock {\em Proc. Japan Acad. Ser. A Math. Sci.}, 72(10):240–241,   1996.

\bibitem{IM} Nikolai V Ivanov, John D McCarthy.
\newblock On injective homomorphisms between Teichm{u}ller modular groups. I
\newblock {\em Invent. Math.}, 135:425–486, 1999.

\bibitem{KS} A. Karrass and D. Solitar.
\newblock Note on a theorem of Schreier.
\newblock {\em Proc. Amer. Math.Soc.}, 8:696-697, 1957.

\bibitem{KL} C. Kurth and L.Long
\newblock Computations with finite index subgroups of PSL2(Z) using Farey Symbols.
\newblock {\em Advances in Algebra and Combinatorics},  225-242, 2008.

\bibitem{Ku} R. S. Kulkarni,
\newblock An arithmetic geometric method in the study of the subgroups of the modular group.
\newblock {\em American Journal of Mathematics}, 113, 6:1053–1133, 1991.

 
 \bibitem{L} 
 W. B. R. Lickorish. 
\newblock A finite set of generators for the homeotopy group of a 2-manifold.
\newblock {\em Proc. Cambridge Philos. Soc. }, 60:769–778, 1964.


 \bibitem{N} E. Newman. 
\newblock Integral matrices.
\newblock {\em Academic Press}, 1972.
 
 \bibitem{Sa} L. N. Sanov. 
 \newblock  A property of a representation of a free group.
 \newblock {\em Douk. Acad. Nauk.}, 657-659, 1957.
 

 
\bibitem{S} Martin Scharlemann.  
\newblock Subgroups of SL(2,R) freely generated by three parabolic elements.
\newblock {\em  Linear and Multilinear Algebra}, 7, 177-191, 1979.

 \bibitem{Sc}  O.  Schreier, 
 \newblock Die Untergruppen der  freien Gruppen, 
 \newblock {\em Abhandlungen aus  dem Mathematischen Seminar der  Universitat Hamburg.}, vol.  5 
  161, 1928.
 

 


\end{thebibliography}
\end{document}